\documentclass[mathpazo,reqno]{./lib/template}
\usepackage[margin=1in]{geometry}
\usepackage[colorlinks,linkcolor=blue,citecolor=blue]{hyperref} 
\usepackage{amsmath,amssymb,amsxtra,mathrsfs,amsfonts,amsthm,bm,cases,xcolor}
\usepackage{booktabs,multirow}
\usepackage{algorithm,algorithmic}
\usepackage{graphicx,subfigure,float}
\usepackage{makecell}
\usepackage{algorithm}
\usepackage{algorithmic}
\usepackage[colorinlistoftodos]{todonotes}
\usepackage[bb=boondox]{mathalfa}
\usepackage{hyperref}
\hypersetup{
	colorlinks=true,
	urlcolor=blue  % 
}

\usepackage{url}  % 处理 URL 格式
\def\b0{\boldsymbol{0}}

\def\bu{\boldsymbol{u}}

\def\wh{\widehat}

\def\bbA{\mathbb{A}}
\def\bbB{\mathbb{B}}
\def\bbC{\mathbb{C}}
\def\bbD{\mathbb{D}}
\newtheorem{sch}{Scheme}[section]
\newtheorem{thm}{Theorem}[section]

\newtheorem{lem}{Lemma}[section]

\newtheorem{rem}{Remark}[section]

\usepackage{caption}
\begin{document}
	%%%%% title : short title may not be used but TITLE is required.
	% \title{TITLE}
	% \title[short title]{TITLE}
	\title{Original-energy-dissipation-preserving methods for the incompressible Navier-Stokes equations}
	
	%%%%% author(s) :
	\author[Weng ZH. et.~al.]{Zihan Weng,
		Qi Hong, Chunwu Wang,
		Yuezheng Gong\comma\corrauth}
	\address{School of Mathematics, Nanjing University of Aeronautics and Astronautics, Key Laboratory of Mathematical Modelling and High Performance Computing of Air Vehicles (NUAA), MIIT, Nanjing, Jiangsu  211106, China}
	\email{{\tt gongyuezheng@nuaa.edu.cn} (Y.~Gong)
	}
	%=========================
	% \footnote and \thanks are not used in the heading section.
	% Another acknowlegments/support of grants, state in Acknowledgments section
	% \section*{Acknowledgments}

	%%%%% Begin Abstract %%%%%%%%%%%
	\begin{abstract}
		This paper introduces a robust reformulation of the incompressible Navier-Stokes equations, establishing a foundational framework for designing efficient, structure-preserving algorithms that strictly conserve the original energy dissipation law. By leveraging Crank-Nicolson schemes and backward differentiation formulas, we develop four first- and second-order time-discrete schemes. These schemes exactly preserve the original energy dissipation law at each time step, requiring only the solutions of three linear Stokes systems and one $2\times 2$ system of linear equations. Furthermore, the finite difference approximation on a staggered grid is employed for these time-discrete systems to derive fully discrete structure-preserving schemes. We rigorously prove that all proposed fully discrete methods both maintain the original energy dissipation law and admit unique solutions.  Moreover, we present their efficient implementation. Extensive numerical experiments are carried out to verify the accuracy, efficacy, and advantageous performance of our newly developed methods.
	\end{abstract}
	%%%%% end %%%%%%%%%%

	%%%%% AMS/PACs/Keywords %%%%%%%%%%%
	%\pac{}
	%\ams{52B10, 65D18, 68U05, 68U07}
	
	
	%%%% maketitle %%%%%
	\maketitle
	%%%% Start %%%%%%
	\section{Introduction}
	\label{sec1} 
	As fundamental equations in computational fluid dynamics, the incompressible Navier-Stokes (NS) equations describe the dynamics of incompressible fluid flows. In this paper, we focus on their dimensionless formulation:
	\begin{equation}\label{incompressibleNSEs}
		\begin{cases}
			\bm{u}_t - \nu \Delta \bm{u} + \bm{u} \cdot \nabla \bm{u} + \nabla p = \bm{f},& \text{in} ~\Omega \times (0,T],\\
			\nabla \cdot \bm{u} = 0,& \text{in} ~\Omega \times (0,T],
		\end{cases}
	\end{equation}
	subject to an initial condition $\bm{u}(\bm{x},0) = \bm{u}_0(\bm{x})$ and either homogeneous Dirichlet or periodic boundary conditions within a rectangular domain $\Omega\subset \mathbb{R}^d~(d=2,3)$.  Here, the primary unknowns are the velocity field $\bm{u}$ and pressure $p$, with the latter constrained by its zero mean value to ensure uniqueness. The term $\bm{f}$ denotes external body forces, while $\nu>0$ represents the kinematic viscosity coefficient which is inversely proportional to the Reynolds number $Re$. It is well known that, in the absence of external force ($\bm{f} = \bm{0}$), the incompressible NS equations \eqref{incompressibleNSEs} under the prescribed boundary conditions satisfy the energy dissipation law
	\begin{equation}\label{EDLf}
		\frac{\mathrm{d}}{\mathrm{d}t} \int_\Omega \frac{1}{2} |\bm{u}|^2 \mathrm{d} \bm{x} = - \nu \int_{\Omega}|\nabla \bm{u}|^2 \mathrm{d} \bm{x}.
	\end{equation}
	Consequently, developing numerical schemes that preserve this essential physical law constitutes a major objective in computational methods for the incompressible NS equations.
	
	A large number of numerical algorithms have been developed to solve the incompressible NS equations \cite{deville2002high,gunzburger2012finite,temam2024navier,peyret2002spectral}. Among these, the projection method---also known as the fractional step method---stands out as a widely adopted approach since its introduction by Chorin \cite{ChorinNumerical1968} and Temam \cite{TemamARMA1969}, owing to its simplicity and efficiency \cite{KIM1985JCP,GUERMONDCMAME2006,ESIAM1995,EIJNM2000,TIMMERMANS1996}. However, pursuing extreme computational efficiency by applying explicit treatment to the convective term introduces stringent stability constraints on the time step, which compromise both the efficacy and robustness of the scheme \cite{Li2021MC}. Notably, it remains entirely feasible to construct unconditionally stable schemes within the projection framework, as exemplified by the scheme (5.1)-(5.2) in \cite{Shen1992SINUM}. Nevertheless, achieving such stability necessitates semi-implicit discretization of the convective term, requiring the solution of a system of linear algebraic equations with a variable and time-dependent coefficient matrix. Furthermore, we emphasize that these projection schemes do not conserve the discrete energy dissipation law \eqref{EDLf}.
	
	Numerical methods that can preserve some structural properties of the model are called geometric integrators or structure-preserving algorithms \cite{Hairer2006geometric}. There is no question that a structural preserving numerical scheme is a token of success in the numerical approximation. For the incompressible NS equations, specialized discretizations exist wherein the discrete counterpart of the energy dissipation law holds rigorously at each time step \cite{SIMO1994,GONG2016JSC,GONG2018SIAM}. While most existing approaches focus exclusively on temporal discretization, constructing fully discrete structure-preserving algorithms demands careful selection of compatible spatial discretization schemes. To develop structure-preserving spatial discretizations using the finite difference method, Verstappen and Veldman proposed the symmetry-preserving discretization for the incompressible NS equations, which could further produces fully discrete structure-preserving methods \cite{VERSTAPPEN2003JCP}. However, these formulations suffer from computational inefficiency due to the inherent coupling between velocity and pressure fields, inevitably necessitating the solution of a system of nonlinear algebraic equations or a system of linear algebraic equations with a variable and time-dependent coefficient matrix.
	
	Motivated by the recently proposed scalar auxiliary variable (SAV) approach for gradient flows \cite{SHENJCP2018,shensiam2019}, Lin, Yang and Dong devised an efficient, unconditionally energy stable scheme for the incompressible NS equations \cite{LINJCP2019}. Their method explicitly discretizes the nonlinear convective term while preserving a modified energy dissipation law, marking the first time that such stability was achieved through explicit treatment of the nonlinear convection term. Subsequently, building on this foundation, Li and Shen constructed and analyzed a fully discrete scheme using the SAV approach for time integration coupled with the Marker-and-Cell method for spatial discretization \cite{SHEN2021SIAM}. Moreover, even more efficient decoupled energy stable schemes for the NS equations have been derived by combining the SAV approach with the projection method, where only a sequence of Poisson-type equations needs to be solved at each time step \cite{Li2021MC,WUJCP2022,ObbadiCMAME2025}. Beyond the SAV framework, it is well established that the Lagrange multiplier technique also yields unconditionally energy stable schemes. Following the work of Cheng et al.  \cite{CHENG2020cmame}, Yang and colleagues introduced a Lagrange multiplier directly into the dynamic equation governing the original kinetic energy in their study \cite{YANGCF2022}. The proposed second-order backward differentiation formula (BDF) scheme was demonstrated to satisfy an energy stability property encompassing both the original energy and specific pressure gradient terms. More recently, Doan et al. developed a dynamically regularized Lagrange multiplier (DRLM) method, which requires only the solutions of two linear Stokes systems and a scalar quadratic equation at each time step \cite{doan4938899dynamically}. With the introduction of the regularization parameter, the Lagrange multiplier can be uniquely determined from the quadratic equation, even with large time step sizes, without affecting accuracy and stability of the numerical solutions. More interestingly, Yang came up with the novel concept of zero energy contribution (ZEC) and went on to develop highly efficient, energy stable schemes for various complex fluid phase field models \cite{YANG2020WILEY,YANG2021JCP,YANG2021CMAMESURF}.
	
	While both the SAV and Lagrange multiplier approaches offer efficient implementation, requiring only the solutions of generalized Stokes systems (for the coupled approach) or a sequence of Poisson-type equations (for the decoupled approach), they usually pose three challenges: (i) Solving the nonlinear algebraic equations for the Lagrange multiplier at each time step can result in multiple real or even complex solutions; (ii) The time step must be sufficiently small to ensure accuracy, especially for the Lagrange multiplier method, leading to high computational costs for long-term simulations; (iii) Most stability results follow a modified energy dissipation law associated with auxiliary variables, rather than the original physical one.
	
	In this paper, our objective is to focus on developing novel, efficient structure-preserving algorithms for the incompressible NS equations to overcome the aforementioned challenges. We begin by introducing a new reformulated technique, termed the robust reformulation, for the incompressible NS equations. This approach yields a strongly equivalent system to the original model, entirely free of auxiliary variables.  A key advantage of this reformulation is its provision of an ideal platform for explicitly treating the convective term while preserving the original energy dissipation law. Based on the Crank-Nicolson (CN) and BDF methods, we present four first- and second-order time-discrete schemes, which are shown to conserve the discrete original energy dissipation law and require only the solutions of three generalized Stokes systems and a $2\times 2$ system of linear equations at each time step. Furthermore, we derive fully discrete structure-preserving schemes using finite differences on staggered grids for spatial discretization \cite{GONG2018SIAM,TC2023Hong,doan4938899dynamically}. We rigorously prove that these proposed methods both conserve the original energy dissipation law at the fully discrete level and guarantee unique solvability. Their efficient implementation is detailed comprehensively. To validate our newly developed schemes, we conduct numerical experiments including a manufactured convergence test and simulations of the Taylor-Green vortex problem, confirming convergence and energy dissipation properties. Finally, we demonstrate the practical applicability of our schemes in capturing dynamic evolution in realistic scenarios, such as lid-driven cavity flow and Kelvin-Helmholtz instability.
	
	The rest of this paper is organized as follows. In Section \ref{Model refomulation}, we introduce our robust reformulation for the incompressible NS equations. Highly efficient time-discrete schemes are derived and shown to preserve the original energy dissipation law in Section \ref{time discretization}. Fully discrete structure-preserving schemes with the staggered-grid finite difference discretization are presented and studied in Section \ref{Fully discrete schemes}.  In Section \ref{Numerical experiments}, various numerical experiments are conducted to demonstrate the accuracy, efficiency, and structure-preserving properties of the newly developed schemes. Finally, we give some conclusions in Section \ref{Conclusion}.
	
	\section{Robust reformulation}\label{Model refomulation}
	In this section, we present a robust reformulation of the incompressible NS equations. This reformulation enables designing highly efficient, linearly energy stable schemes while strictly preserving the original energy dissipation law. Throughout this paper, the standard $L^2$ inner product and the corresponding $L^2$ norm are denoted by $(\cdot,\cdot)$ and $\| \cdot \|$, respectively.
	
	For an incompressible velocity field $\bm{u}$ (i.e., $\nabla \cdot \bm{u} = 0$), subject to either homogeneous Dirichlet or periodic boundary conditions, the following identity holds
	\begin{equation}\label{convectproperty}
		(\bm{u}\cdot \nabla \bm{u}, \bm{u})=0.
	\end{equation}
	Assuming that $\bm{F}(\bm{u})$ is a user-specified function satisfying the non-degeneracy condition
	\begin{equation}\label{eqnonneg}
		\big(\bm{F}(\bm{u}), \bm{u}\big) \neq 0, \quad \forall \|\bm{u}\| \neq 0,
	\end{equation}
	we introduce two functions, $\bm{G}(\bm{u})$ and $\bm{B}(\bm{u},\bm{v})$, defined as follows
	\begin{equation}\label{eqG}
		\bm{G}(\bm{u}) = \begin{cases}
			\frac{\bm{u}\cdot \nabla \bm{u}}{\big(\bm{F}(\bm{u}), \bm{u}\big)}, & \|\bm{u}\| \neq 0,\\
			\bm{0}, & \text{otherwise},
		\end{cases}
	\end{equation}
	and
	\begin{align}\label{eqG2}
		\bm{B}(\bm{u},\bm{v})=\big(\bm{F}(\bm{u}), \bm{v} \big) \bm{G}(\bm{u}) - \big(\bm{G}(\bm{u}), \bm{v} \big) \bm{F}(\bm{u}).
	\end{align}
	It is not difficult to find that $\bm{B}(\bm{u},\bm{v})$ is linear with respect to the second variable $\bm{v}$. Next we give an useful lemme.
	\begin{lem}\label{convectproperty2}
		Under homogeneous Dirichlet or periodic boundary conditions, it holds for $\nabla \cdot \bm{u}=0$ that
		\begin{equation}\label{key_convectproperty1}
			\bm{B}(\bm{u},\bm{u})=\bm{u}\cdot \nabla \bm{u} , 
		\end{equation}
		and
		\begin{equation}\label{key_convectproperty2}
			\big(\bm{B}(\bm{u},\bm{v}),\bm{v}\big)=0.
		\end{equation}
	\end{lem}
	
	\begin{proof}
		According to Eqs. \eqref{convectproperty} and \eqref{eqG}, we can obtain
		$$ \big(\bm{F}(\bm{u}), \bm{u}\big) \bm{G}(\bm{u}) = \bm{u}\cdot \nabla \bm{u},$$ and $$\big(\bm{G}(\bm{u}), \bm{u}\big) = 0,$$
		which can directly lead to
		\begin{equation}
			\bm{B}(\bm{u},\bm{u})=\big(\bm{F}(\bm{u}), \bm{u} \big) \bm{G}(\bm{u}) - \big(\bm{G}(\bm{u}), \bm{u} \big) \bm{F}(\bm{u})=\bm{u}\cdot \nabla \bm{u}.
		\end{equation}
		Moreover, taking the $L^2$ inner product of \eqref{eqG2} with $\bm{v}$, we arrive at
		\begin{equation}
			\big(\bm{B}(\bm{u},\bm{v}),\bm{v}\big)=\big(\bm{F}(\bm{u}), \bm{v} \big)\cdot \big(\bm{G}(\bm{u}),\bm{v}\big) - \big(\bm{G}(\bm{u}), \bm{v} \big)\cdot \big(\bm{F}(\bm{u}),\bm{v}\big)=0.
		\end{equation}
	\end{proof}
	
	Consequently, when subject to homogeneous Dirichlet or periodic boundary conditions, the incompressible NS system \eqref{incompressibleNSEs} admits an equivalent robust reformulation
	\begin{subequations}\label{equalsys}
		\begin{numcases}{}
			\bm{u}_t - \nu \Delta \bm{u} + \bm{B}(\bm{u},\bm{u}) + \nabla p = \bm{f} ,\label{equalsys1}  \\ [0.15cm]
			\nabla \cdot \bm{u} = 0.
		\end{numcases}
	\end{subequations}
	Significantly, this robust reformulation provides a powerful and elegant platform for developing highly efficient, linearly energy stable schemes that faithfully preserve the original energy dissipation law. The practical construction of such schemes will be detailed in the ensuing sections. Here, we explain at the PDE level why linearly structure-preserving algorithms arise very naturally from this robust variant. Starting from an explicit approximation $\overline{\bm{u}}$ of $\bm{u}$, we replace $\bm{B}(\bm{u},\bm{u})$ in \eqref{equalsys} with $\bm{B}(\overline{\bm{u}},\bm{u})$ and then obtain the following perturbed version
	\begin{subequations}\label{linearized_system}
		\begin{numcases}{}
			\bm{u}_t - \nu \Delta \bm{u} + \bm{B}(\overline{\bm{u}},\bm{u}) + \nabla p = \bm{f} ,\label{LS1}  \\ [0.15cm]
			\nabla \cdot \bm{u} = 0.
		\end{numcases}
	\end{subequations}
	Notably, since $\bm{B}(\overline{\bm{u}},\bm{u})$ is linearly dependent on $\bm{u}$, the perturbed system \eqref{linearized_system} forms a coupled linear PDE framework governing both $\bm{u}$ and $p$. Hence, we can also regard \eqref{linearized_system} as the linearization of \eqref{equalsys}. 
	
	\begin{thm}\label{thm:EDL}
		Under homogeneous Dirichlet or periodic boundary conditions, and in the absence of external force $\bm{f}$, the linearized system \eqref{linearized_system} still satisfies the original energy dissipation law  
		\begin{align}
			\frac{\mathrm{d}}{\mathrm{d}t} E = -\nu \|\nabla \bm{u}\|^2, \quad E = \frac{1}{2}\|\bm{u}\|^2.
		\end{align}
	\end{thm}
	
	\begin{proof}
		Under homogeneous Dirichlet or periodic boundary conditions, it is readily to check that
		$$(\Delta \bm{u}, \bm{u}) = -\|\nabla \bm{u}\|^2 \leq 0, \quad (\nabla p, \bm{u}) = -(p, \nabla \cdot \bm{u}) = 0,$$
		where $\nabla \cdot \bm{u} = 0$ was used. According to Lemme \ref{convectproperty2}, it is clear that 
		$$ \big(\bm{B}(\overline{\bm{u}},\bm{u}),\bm{u}\big)=0.$$ For the external force $\bm{f}=\bm{0},$ taking the $L^2$ inner product of \eqref{LS1} with $\bm{u}$, we obtain
		$$(\bm{u}_t, \bm{u}) = -\nu\|\nabla \bm{u}\|^2.$$
		Consequently, we have
		$$\frac{\mathrm{d}}{\mathrm{d}t} E = (\bm{u}_t, \bm{u}) = -\nu\|\nabla \bm{u}\|^2.$$
	\end{proof}
	
	\begin{rem}
		According to Theorem \ref{thm:EDL}, our focus shifts to designing structure-preserving algorithms for the linearized system \eqref{linearized_system}. Such algorithms are not only easy to construct but also naturally linear. Remarkably, the schemes derived from this procedure can be implemented efficiently, as demonstrated in the following sections.
	\end{rem}
	
	\begin{rem}
		For general inhomogeneous Dirichlet boundary conditions $\bm{u}|_{\partial \Omega}=\bm{u}_b$, it holds for $\nabla \cdot \bm{u}=0$ that \cite{doan4938899dynamically}
		\begin{equation}
			(\bm{u}\cdot \nabla \bm{u},\bm{u})=\frac{1}{2} (\bm{u}_b\cdot \bm{n} , |\bm{u}_b|^2)_{\partial \Omega},
		\end{equation}
		where $(\cdot, \cdot)_{\partial \Omega}$ denotes the $L^2$ inner product on $\partial \Omega$ and $\bm{n}$ represents the outward unit normal vector to $\partial \Omega.$ In this case, it suffices to modify $\bm{B}(\bm{u},\bm{v})$ as
		\begin{equation}
			\bm{B}(\bm{u},\bm{v}) = \begin{cases}
				\big(\bm{F}(\bm{u}), \bm{v}\big) \bm{G}(\bm{u}) - \frac{(\bm{u}\cdot \nabla \bm{u},\bm{v})-\frac{1}{2} (\bm{u}_b\cdot \bm{n} , |\bm{u}_b|^2)_{\partial \Omega}}{\big(\bm{F}(\bm{u}), \bm{u}\big)} \bm{F}(\bm{u}), & \|\bm{u}\| \neq 0,\\
				\bm{0}, & \text{otherwise},
			\end{cases}
		\end{equation}
		while the numerical methods proposed in this paper remain valid.
	\end{rem}
	
	\section{Highly efficient time-discrete schemes}\label{time discretization}
	In this section, we present some highly efficient time-discrete schemes that rigorously preserve the original energy dissipation law. Let us next assume a uniform partition of the time interval $[0,T]:~ 0=t_0<t_1<\cdots<t_n<t_{n+1}<\cdots<t_{N_t}=T$ with the time step size $\tau = T/N_t$ and consider the time discretization of the robust equivalent model \eqref{equalsys}.
	
	\subsection{Crank-Nicolson schemes}
	Firstly, we present a first-order linear scheme based on Crank-Nicolson (CN) method, that reads as follows.
	
	\begin{sch}[CN1]\label{sch:CN1}
		Given the initial condition $\bm{u}^0$, we compute $\bm{u}^{n+1}$ for $0\leq n \leq N_t-1$ via
		\begin{subequations}\label{CN1scheme}
			\begin{numcases}{}
				\frac{\bm{u}^{n+1}-\bm{u}^n}{\tau} -\nu \Delta \bm{u}^{n+\frac{1}{2}} +\bm{B}(\bm{u}^{n},\bm{u}^{n+\frac{1}{2}}) + \nabla p^{n+\frac{1}{2}}=\bm{f}^{n+\frac{1}{2}},\label{CN1scheme1}  \\[0.15cm]
				\nabla \cdot \bm{u}^{n+\frac{1}{2}}=0,
			\end{numcases}
		\end{subequations}
		where $\bm{u}^{n+\frac{1}{2}}=(\bm{u}^n+\bm{u}^{n+1})/2$, and $\bm{u}^{n+1}$ satisfies either the homogeneous Dirichlet or periodic boundary conditions.
	\end{sch}
	
	To construct the second-order linear CN scheme, we assume $\bm{u}^{n-1}$ and $\bm{u}^{n}$ are given and denote $\overline{\bm{u}}^{n+\frac{1}{2}}=(3\bm{u}^{n}-\bm{u}^{n-1})/2$ as the second-order extrapolation for approximating $\bm{u}(t_{n+\frac{1}{2}}).$ Then we develop the following second-order linear CN scheme.
	
	\begin{sch}[CN2]\label{sch:CN2}
		Given the initial condition $\bm{u}^0$, and $\bm{u}^1$ is obtained using the CN1 scheme, we compute $\bm{u}^{n+1}$ for $1\leq n \leq N_t-1$ via
		\begin{subequations}\label{CN2scheme}
			\begin{numcases}{}
				\frac{\bm{u}^{n+1}-\bm{u}^n}{\tau} -\nu \Delta \bm{u}^{n+\frac{1}{2}} + \bm{B}(\overline{\bm{u}}^{n+\frac{1}{2}},\bm{u}^{n+\frac{1}{2}}) + \nabla p^{n+\frac{1}{2}}=\bm{f}^{n+\frac{1}{2}}, \label{CN2scheme1}  \\[0.15cm]
				\nabla \cdot \bm{u}^{n+\frac{1}{2}}=0,
			\end{numcases}
		\end{subequations}
		where $\bm{u}^{n+1}$ satisfies either the homogeneous Dirichlet or periodic boundary conditions.
	\end{sch}
	
	\begin{thm}\label{thm:CN_EDL}
		Under homogeneous Dirichlet or periodic boundary conditions, and in the absence of external force $\bm{f}$, both the CN1 and CN2 schemes preserve the following discrete energy dissipation law
		\begin{align}
			\frac{E^{n+1}-E^n}{\tau} = -\nu \|\nabla \bm{u}^{n+\frac{1}{2}} \|^2 ,
		\end{align}
		where the discrete energy at $t = t_n$ is defined as
		\begin{align}
			E^{n} = \frac{1}{2}\|\bm{u}^n\|^2.
		\end{align}
	\end{thm}
	
	\begin{proof}
		Similar to the proof of Theorem \ref{thm:EDL}, taking the $L^2$ inner product of \eqref{CN1scheme1} or \eqref{CN2scheme1} with $\bm{u}^{n+\frac{1}{2}}$, one can arrive at
		\begin{align*}
			\left(\frac{\bm{u}^{n+1}-\bm{u}^{n}}{\tau},\bm{u}^{n+\frac{1}{2}}\right) = -\nu \|\nabla \bm{u}^{n+\frac{1}{2}} \|^2.
		\end{align*}
		Therefore, we have
		\begin{align*}
			\frac{E^{n+1}-E^n}{\tau} =	\left(\frac{\bm{u}^{n+1}-\bm{u}^{n}}{\tau},\bm{u}^{n+\frac{1}{2}}\right) = -\nu \|\nabla \bm{u}^{n+\frac{1}{2}} \|^2.
		\end{align*}
	\end{proof}
	
	In what follows, we present an efficient implementation of the CN2 scheme. The CN1 scheme can be solved similarly, and the details are omitted for brevity. For the sake of simplicity, we introduce the following notations
	\begin{align}\label{eq:alpha-beta}
		\alpha = \big(\bm{F}(\overline{\bm{u}}^{n+\frac{1}{2}}), \bm{u}^{n+\frac{1}{2}} \big), \quad
		\beta = \big(\bm{G}(\overline{\bm{u}}^{n+\frac{1}{2}}), \bm{u}^{n+\frac{1}{2}} \big).
	\end{align}
	It is readily to check that
	\begin{align}
		\bm{B}(\overline{\bm{u}}^{n+\frac{1}{2}},\bm{u}^{n+\frac{1}{2}})=\alpha \bm{G}(\overline{\bm{u}}^{n+\frac{1}{2}})-\beta \bm{F}(\overline{\bm{u}}^{n+\frac{1}{2}}).
	\end{align}
	Therefore, the linear system \eqref{CN2scheme} can be rewritten into 
	\begin{subequations}\label{Compactformsheme1st}
		\begin{numcases}{}
			\frac{2(\bm{u}^{n+\frac{1}{2}}-\bm{u}^n)}{\tau}-\nu \Delta \bm{u}^{n+\frac{1}{2}} +\alpha \bm{G}(\overline{\bm{u}}^{n+\frac{1}{2}})-\beta \bm{F}(\overline{\bm{u}}^{n+\frac{1}{2}})+\nabla p^{n+\frac{1}{2}}=\bm{f}^{n+\frac{1}{2}},  \\[0.15cm]
			\nabla \cdot \bm{u}^{n+\frac{1}{2}}=0,
		\end{numcases}
	\end{subequations}
	where $\bm{u}^{n+1} = 2\bm{u}^{n+\frac{1}{2}} - \bm{u}^n$ was used.
	It is readily to check that the unknown quantities $\bm{u}^{n+\frac{1}{2}}$ and $p^{n+\frac{1}{2}}$ can be decomposed into three components as follows
	\begin{subequations}\label{Decomposedtimed}
		\begin{numcases}{}	
			\bm{u}^{n+\frac{1}{2}} = \alpha \bm{u}^{n+\frac{1}{2}}_1+\beta \bm{u}^{n+\frac{1}{2}}_2+\bm{u}^{n+\frac{1}{2}}_3, \label{Decomposedtimed1}\\
			p^{n+\frac{1}{2}}=\alpha p^{n+\frac{1}{2}}_1+\beta p^{n+\frac{1}{2}}_2+p^{n+\frac{1}{2}}_3,\label{Decomposedtimed2}
		\end{numcases}
	\end{subequations}
	where $\bm{u}^{n+\frac{1}{2}}_i$ and $p^{n+\frac{1}{2}}_i \ (i=1,2,3)$ yield three systems of generalized Stokes equations, namely,
	\begin{subequations}\label{Decomposedsystem}
		\begin{numcases}{}
			\frac{2}{\tau}\bm{u}^{n+\frac{1}{2}}_i-\nu \Delta \bm{u}^{n+\frac{1}{2}}_i +\nabla p^{n+\frac{1}{2}}_i=\bm{g}^n_i,  \\[0.15cm]
			\nabla \cdot \bm{u}^{n+\frac{1}{2}}_i=0 ,
		\end{numcases}
	\end{subequations}
	with
	\begin{align*}
		\bm{g}^n_1=-\bm{G}(\overline{\bm{u}}^{n+\frac{1}{2}}), \quad
		\bm{g}^n_2=\bm{F}(\overline{\bm{u}}^{n+\frac{1}{2}}),\quad
		\bm{g}^n_3=\frac{2\bm{u}^n}{\tau}+\bm{f}^{n+\frac{1}{2}}.
	\end{align*}
	It is worth noting that $\bm{u}^{n+\frac{1}{2}}_i \ (i=1,2,3)$ satisfy the same boundary conditions as $\bm{u}^{n+\frac{1}{2}}$, i.e., either homogeneous Dirichlet or periodic boundary conditions. Moreover, there exists numerous well-established and efficient methods for solving the linear system \eqref{Decomposedsystem}, which we do not elaborate on here. Interested readers are encouraged to consult Refs. \cite{Boyce2009JCP,doan4938899dynamically} for further details. Once $\bm{u}^{n+\frac{1}{2}}_i \ (i=1,2,3)$ are determined, plugging \eqref{Decomposedtimed1} into \eqref{eq:alpha-beta} leads to the following $2 \times 2$ linear system with respect to $\alpha$ and $\beta$
	\begin{align}\label{Matrixequation}
		A x =	b,
	\end{align}
	where 
	\begin{align*}
		A = \begin{pmatrix}
			1-(\bm{F}(\overline{\bm{u}}^{n+\frac{1}{2}}),\bm{u}^{n+\frac{1}{2}}_1) & -(\bm{F}(\overline{\bm{u}}^{n+\frac{1}{2}}),\bm{u}^{n+\frac{1}{2}}_2)\\
			-(\bm{G}(\overline{\bm{u}}^{n+\frac{1}{2}}),\bm{u}^{n+\frac{1}{2}}_1) &  1-(\bm{G}(\overline{\bm{u}}^{n+\frac{1}{2}}),\bm{u}^{n+\frac{1}{2}}_2)
		\end{pmatrix},\quad
		x = \begin{pmatrix}
			\alpha\\
			\beta
		\end{pmatrix},\quad b = \begin{pmatrix}
			(\bm{F}(\overline{\bm{u}}^{n+\frac{1}{2}}),\bm{u}^{n+\frac{1}{2}}_3)\\
			(\bm{G}(\overline{\bm{u}}^{n+\frac{1}{2}}),\bm{u}^{n+\frac{1}{2}}_3)
		\end{pmatrix}.
	\end{align*}
	After solving $\alpha$ and $\beta$, $\bm{u}^{n+\frac{1}{2}}$ and $p^{n+\frac{1}{2}}$ can be updated by utilizing \eqref{Decomposedtimed1} and \eqref{Decomposedtimed2}, respectively. Finally, we compute 
	\begin{align}\label{UpdatefinalCN}
		\bm{u}^{n+1} = 2\bm{u}^{n+\frac{1}{2}} - \bm{u}^n.
	\end{align}
	To assist readers who are interested, we provide a summary of the implementations of our proposed CN2 scheme in Algorithm \ref{power}.
	\begin{algorithm}[H]
		\renewcommand{\algorithmicrequire}{\textbf{Date:}}
		\renewcommand{\algorithmicensure}{\textbf{Result:}}
		\caption{Pseudocode outlining the implementation of Scheme \ref{CN2scheme} for each time step}
		\label{power}
		\begin{algorithmic}[1] % 控制是否有序号
			\REQUIRE  Input $\bm{u}^{n-1}, ~\bm{u}^n$; % input 的内容		
			\STATE Calculate $\bm{u}_i^{n+\frac{1}{2}}$ and $p_i^{n+\frac{1}{2}}$ from the linear system \eqref{Decomposedsystem};
			\STATE Update $\alpha$ and $\beta$ by solving the $2\times 2$ linear system \eqref{Matrixequation};
			\STATE Calculate $\bm{u}^{n+\frac{1}{2}}$ and $p^{n+\frac{1}{2}}$ using \eqref{Decomposedtimed1} and \eqref{Decomposedtimed2}, respectively;
			\STATE Update $\bm{u}^{n+1}$ using \eqref{UpdatefinalCN};
			\ENSURE Output $\bm{u}^{n+1}$. % output 的内容
		\end{algorithmic}
	\end{algorithm}
	
	\subsection{Backward differentiation formula schemes}
	By applying the first-order backward differentiation formula (BDF) method to the robust equivalent model \eqref{equalsys}, we propose the following linear BDF1 scheme.
	\begin{sch}[BDF1]\label{sch:BDF1}
		Given the initial condition $\bm{u}^0$, we compute $\bm{u}^{n+1}$ for $0\leq n \leq N_t-1$ via
		\begin{subequations}\label{BDF1scheme}
			\begin{numcases}{}
				\frac{\bm{u}^{n+1}-\bm{u}^n}{\tau} -\nu \Delta \bm{u}^{n+1} + 	\bm{B}(\bm{u}^n,\bm{u}^{n+1}) + \nabla p^{n+1}=\bm{f}^{n+1},\label{BDF1scheme1}  \\[0.15cm]
				\nabla \cdot \bm{u}^{n+1}=0,
			\end{numcases}
		\end{subequations}
		where $\bm{u}^{n+1}$ satisfies either the homogeneous Dirichlet or periodic boundary conditions.
	\end{sch}
	\begin{thm}
		Under homogeneous Dirichlet or periodic boundary conditions, and in the absence of external force $\bm{f}$, the BDF1 scheme preserves the following discrete energy dissipation law
		\begin{align}
			\frac{E^{n+1}-E^n}{\tau} = -\nu \|\nabla \bm{u}^{n+1} \|^2 -\frac{1}{2\tau}\| \bm{u}^{n+1}-\bm{u}^n \|^2,
		\end{align}
		where the discrete energy at $t = t_n$ is defined as
		\begin{align}
			E^{n} = \frac{1}{2}\|\bm{u}^n\|^2.
		\end{align}
	\end{thm}
	
	\begin{proof}
		Similar to the proof of Theorem \ref{thm:EDL}, taking the $L^2$ inner product of \eqref{BDF1scheme1} with $\bm{u}^{n+1}$, one can arrive at
		\begin{align*}
			\left(\frac{\bm{u}^{n+1}-\bm{u}^{n}}{\tau},\bm{u}^{n+1}\right) = -\nu \|\nabla \bm{u}^{n+1} \|^2.
		\end{align*}
		It is easy to check that 
		\begin{align*}
			\left(\bm{u}^{n+1}-\bm{u}^{n},\bm{u}^{n+1}\right)=\frac{1}{2} \left(\| \bm{u}^{n+1} \|^2-\| \bm{u}^{n}\|^2+ \| \bm{u}^{n+1} -\bm{u}^{n}\|^2\right).
		\end{align*}
		Therefore, we have
		\begin{align*}
			\frac{E^{n+1}-E^n}{\tau} = -\nu \|\nabla \bm{u}^{n+1} \|^2 -\frac{1}{2\tau}\| \bm{u}^{n+1}-\bm{u}^n \|^2.
		\end{align*}
	\end{proof}
	
	To construct the second-order linear scheme based on the BDF2 method, we assume $\bm{u}^{n-1}$ and $\bm{u}^{n}$ are given and denote $\overline{\bm{u}}^{n+1}=2\bm{u}^{n}-\bm{u}^{n-1}$ as the second-order extrapolation for approximating $\bm{u}(t_{n+1})$. 
	Then we develop the following second-order linear BDF scheme.
	\begin{sch}[BDF2]\label{sch:BDF2}
		Given the initial condition $\bm{u}^0$, and $\bm{u}^1$ is obtained using the BDF1 scheme, we compute $\bm{u}^{n+1}$ for $1\leq n \leq N_t-1$ via
		\begin{subequations}\label{BDF2scheme}
			\begin{numcases}{}
				\frac{3\bm{u}^{n+1}-4\bm{u}^n+\bm{u}^{n-1}}{2\tau} -\nu \Delta \bm{u}^{n+1} + \bm{B}(\overline{\bm{u}}^{n+1},\bm{u}^{n+1}) + \nabla p^{n+1}=\bm{f}^{n+1}, \label{BDF2scheme1}  \\[0.15cm]
				\nabla \cdot \bm{u}^{n+1}=0,
			\end{numcases}
		\end{subequations}
		where $\bm{u}^{n+1}$ satisfies either the homogeneous Dirichlet or periodic boundary conditions.
	\end{sch}
	
	\begin{thm}
		Under homogeneous Dirichlet or periodic boundary conditions, and in the absence of external force $\bm{f}$, the BDF2 scheme preserves the following discrete energy dissipation law
		\begin{align}
			\frac{\wh{E}^{n+1}-\wh{E}^n}{\tau} = -\nu \|\nabla \bm{u}^{n+1} \|^2 -\frac{1}{4\tau}\| \bm{u}^{n+1}-2\bm{u}^n +\bm{u}^{n-1}\|^2,
		\end{align}
		where the discrete energy at $t = t_n$ is defined as
		\begin{align}
			\wh{E}^{n} =\frac{1}{4}\Big(\|\bm{u}^n\|^2+\| 2 \bm{u}^n-\bm{u}^{n-1} \|^2\Big) .
		\end{align}
	\end{thm}
	
	\begin{proof}
		Similar to the proof of Theorem \ref{thm:EDL}, taking the $L^2$ inner product of \eqref{BDF2scheme1} with $\bm{u}^{n+1}$, one can arrive at
		\begin{align*}
			\left(\frac{3\bm{u}^{n+1}-4\bm{u}^{n}+\bm{u}^{n-1}}{2\tau},\bm{u}^{n+1}\right) = -\nu \|\nabla \bm{u}^{n+1} \|^2.
		\end{align*}
		It is readily to check that
		\begin{align*}
			\left(3\bm{u}^{n+1}-4\bm{u}^{n}+\bm{u}^{n-1},2\bm{u}^{n+1}\right)=&\| \bm{u}^{n+1} \|^2-\| \bm{u}^{n}\|^2+ \| 2\bm{u}^{n+1} -\bm{u}^{n}\|^2- \| 2\bm{u}^{n} -\bm{u}^{n-1}\|^2\\ &+ \| \bm{u}^{n+1} -2\bm{u}^{n}+\bm{u}^{n-1}\|^2.
		\end{align*}
		Therefore, we obtain
		\begin{align*}
			\frac{\wh{E}^{n+1}-\wh{E}^n}{\tau} = -\nu \|\nabla \bm{u}^{n+1} \|^2 -\frac{1}{4\tau}\| \bm{u}^{n+1}-2\bm{u}^n +\bm{u}^{n-1}\|^2.
		\end{align*}
	\end{proof}

	In what follows, we present an efficient implementation of the BDF2 scheme. The BDF1 scheme can be solved similarly, and the details are omitted for brevity. For the sake of simplicity, we introduce the following notations
	\begin{align}
		\label{BDFeq:alpha-beta}
		\alpha = \big(\bm{F}(\overline{\bm{u}}^{n+1}), \bm{u}^{n+1} \big), \quad
		\beta = \big(\bm{G}(\overline{\bm{u}}^{n+1}), \bm{u}^{n+1} \big).
	\end{align}
	It is readily to check that
	\begin{align}
		\bm{B}(\overline{\bm{u}}^{n+1},\bm{u}^{n+1})=\alpha \bm{G}(\overline{\bm{u}}^{n+1})-\beta \bm{F}(\overline{\bm{u}}^{n+1}).
	\end{align}
	Therefore, the linear system \eqref{BDF2scheme} can be rewritten into 
	\begin{subequations}\label{BDFCompactformsheme1st}
		\begin{numcases}{}
			\frac{3\bm{u}^{n+1}-4\bm{u}^n+\bm{u}^{n-1}}{2\tau}-\nu \Delta \bm{u}^{n+1} +\alpha \bm{G}(\overline{\bm{u}}^{n+1})-\beta \bm{F}(\overline{\bm{u}}^{n+1})+\nabla p^{n+1}=\bm{f}^{n+1},  \\[0.15cm]
			\nabla \cdot \bm{u}^{n+1}=0 .
		\end{numcases}
	\end{subequations}
	It is readily to check that the unknown quantities $\bm{u}^{n+1}$ and $p^{n+1}$ can be decomposed into three components as follows
	\begin{subequations}\label{BDFDecomposedtimed}
		\begin{numcases}{}	
			\bm{u}^{n+1} = \alpha \bm{u}^{n+1}_1+\beta \bm{u}^{n+1}_2+\bm{u}^{n+1}_3, \label{BDFDecomposedtimed1}\\
			p^{n+1}=\alpha p^{n+1}_1+\beta p^{n+1}_2+p^{n+1}_3,\label{BDFDecomposedtimed2}
		\end{numcases}
	\end{subequations}
	where $\bm{u}^{n+1}_i$ and $p^{n+1}_i \ (i=1,2,3)$ yield three systems of generalized Stokes equations, namely,
	\begin{subequations}\label{BDFDecomposedsystem}
		\begin{numcases}{}
			\frac{3}{2\tau}\bm{u}^{n+1}_i-\nu \Delta \bm{u}^{n+1}_i +\nabla p^{n+1}_i=\bm{g}^n_i,  \\[0.15cm]
			\nabla \cdot \bm{u}^{n+1}_i=0 ,
		\end{numcases}
	\end{subequations}
	with
	\begin{align*}
		\bm{g}^n_1=-\bm{G}(\overline{\bm{u}}^{n+1}), \quad
		\bm{g}^n_2=\bm{F}(\overline{\bm{u}}^{n+1}),\quad
		\bm{g}^n_3=\frac{1}{2\tau}\big(4\bm{u}^n-\bm{u}^{n-1}\big)+\bm{f}^{n+1}.
	\end{align*}
	Once $\bm{u}^{n+1}_i \ (i=1,2,3)$ are determined, plugging \eqref{BDFDecomposedtimed1} into \eqref{BDFeq:alpha-beta} leads to the following $2 \times 2$ linear system with respect to $\alpha$ and $\beta$
	\begin{align}\label{BDFMatrixequation}
		A x =	b,
	\end{align}
	where 
	\begin{align*}
		A = \begin{pmatrix}
			1-(\bm{F}(\overline{\bm{u}}^{n+1}),\bm{u}^{n+1}_1) & -(\bm{F}(\overline{\bm{u}}^{n+1}),\bm{u}^{n+1}_2)\\
			-(\bm{G}(\overline{\bm{u}}^{n+1}),\bm{u}^{n+1}_1) &  1-(\bm{G}(\overline{\bm{u}}^{n+1}),\bm{u}^{n+1}_2)
		\end{pmatrix},\quad
		x = \begin{pmatrix}
			\alpha\\
			\beta
		\end{pmatrix},\quad b = \begin{pmatrix}
			(\bm{F}(\overline{\bm{u}}^{n+1}),\bm{u}^{n+1}_3)\\
			(\bm{G}(\overline{\bm{u}}^{n+1}),\bm{u}^{n+1}_3)
		\end{pmatrix}.
	\end{align*}
	After solving $\alpha$ and $\beta$, $\bm{u}^{n+1}$ and $p^{n+1}$ can be updated by utilizing \eqref{BDFDecomposedtimed1} and \eqref{BDFDecomposedtimed2}, respectively. The implementation procedure of our BDF2 scheme is summarized in Algorithm \ref{power3}.
	\begin{algorithm}[H]
		\renewcommand{\algorithmicrequire}{\textbf{Date:}}
		\renewcommand{\algorithmicensure}{\textbf{Result:}}
		\caption{Pseudocode outlining the implementation of Scheme \ref{sch:BDF2} for each time step}
		\label{power3}
		\begin{algorithmic}[1] % 控制是否有序号
			\REQUIRE  Input $\bm{u}^{n-1},~ \bm{u}^n$; % input 的内容		
			\STATE Calculate $\bm{u}_i^{n+1}$ and $p_i^{n+1}$ from the linear system \eqref{BDFDecomposedsystem};
			\STATE Update $\alpha$ and $\beta$ by solving the $2\times 2$ linear system \eqref{BDFMatrixequation};
			\STATE Calculate $\bm{u}^{n+1}$ and $p^{n+1}$ using \eqref{BDFDecomposedtimed1} and \eqref{BDFDecomposedtimed2}, respectively;
			\ENSURE Output $\bm{u}^{n+1}$. % output 的内容
		\end{algorithmic}
	\end{algorithm}
	
	\section{Fully discrete schemes}\label{Fully discrete schemes}
	In this section, we employ the finite difference approximation on a staggered grid to the time-discrete systems proposed in the previous section, thereby deriving the corresponding fully discrete schemes. For clarity, our discussion centers on the two-dimensional scenario with periodic boundary conditions. However, it is important to mention that the findings of this study are equally relevant to both two- and three-dimensional instances with either homogeneous Dirichlet or periodic boundary conditions. The proposed fully discrete schemes are demonstrated to not only uphold the energy dissipation law at the fully discrete level but also guarantee unique solvability, all while maintaining computational efficiency.
	
	\subsection{The spatial discretization}
	
	Let $\Omega=[0,L_x]\times[0,L_y]$ denote a two-dimensional rectangular domain. Given positive integers $N_x$ and $N_y$, the mesh sizes are defined as $h_x=L_x/N_x$ and $h_y=L_y/N_y$. We further define the following two-dimensional point sets:
	\begin{align*}
		&\Omega_{ew}=\left\{ (x_{i+\frac{1}{2}},y_j)|i=0,\ldots,N_x-1,\; j=1,\ldots,N_y\right\},\\[0.15cm]
		&\Omega_{ns}=\left\{ (x_{i},y_{j+\frac{1}{2}})|i=1,\ldots,N_x,\;j=0,\ldots,N_y-1\right\},\\[0.15cm]
		&\Omega_{c}=\left\{ (x_{i},y_j)|i=1,\ldots,N_x,\;j=1,\ldots,N_y\right\},
	\end{align*}
	where $x_l=(l-1/2)h_x$, $y_l=(l-1/2)h_y$ with $l$ taking either integer or half-integer values. The arrangement of unknown variables within $\Omega_{ew}$, $\Omega_{ns}$ and $\Omega_{c}$ is depicted on a staggered grid in Figure \ref{fig: mesh}. In this schematic, $\Omega_{ew}$ corresponds to blue squares, $\Omega_{ns}$ to red triangles, and $\Omega_{c}$ to black circles. Then we define the following discrete function spaces \cite{GONG2018SIAM,TC2023Hong}
	\begin{align*}
		&V_{ew}=\left\{ U|U=\{u(x_{i+\frac{1}{2}},y_j)|(x_{i+\frac{1}{2}},y_j)\in \Omega_{ew}\} \right\},\\
		&V_{ns}=\left\{V|V=\{v(x_{i},y_{j+\frac{1}{2}})|(x_{i},y_{j+\frac{1}{2}})\in \Omega_{ns} \} \right\},\\
		&V_c=\left\{P|P=\{p(x_{i},y_j)|(x_{i},y_j)\in \Omega_{c}\} \right\}.
	\end{align*}
	For any two matrices $A$ and $B$ of the same size, we define
	\begin{align*}
		(A,B)_h=h_x h_y \mathrm{Tr} (A^T B),\quad \| A \|_h=\sqrt{(A,A)_h},
	\end{align*}
	as the discrete $l^2$ inner product and norm.
	
	\begin{figure}[H]
		\centering
		\includegraphics[width=3.0in]{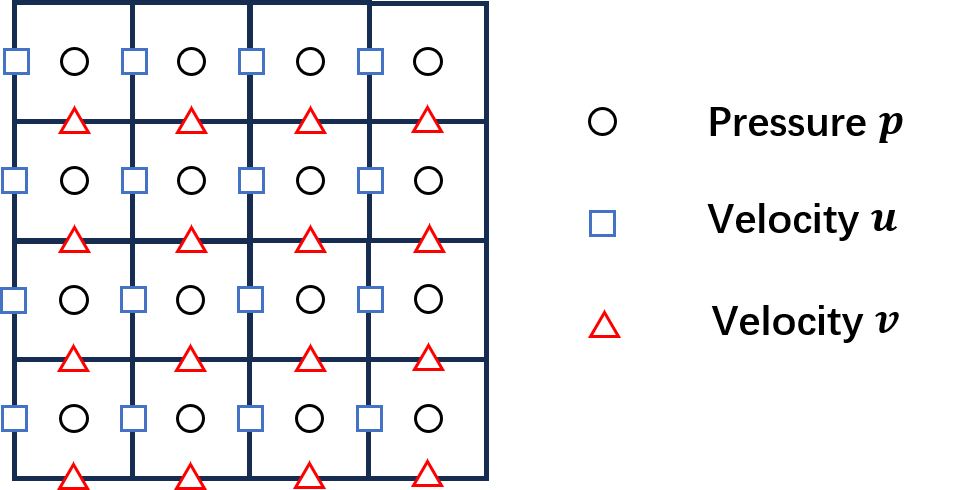}
		\caption{Staggered grid $(N_x=N_y=4)$ with positions of unknowns $\bm{u} = (u, v)^T$ and $p$.}\label{fig: mesh}
	\end{figure}

	Next, we introduce some differentiation matrices as follows:
	\begin{align*}
		&\mathbb{C}(N)=\frac{1}{2}\begin{pmatrix}
			1 & 1 &&& \\
			& 1 &1&&\\
			&  & \ddots &\ddots & \\
			&& &1&1\\
			1&&&&1\\
		\end{pmatrix}_{N\times N}, \quad
		\mathbb{D}_1(N,h)=
		\frac{1}{ 2 h}\begin{pmatrix}
			0 & 1 &&&-1\\
			-1 & 0 &1&&\\
			&  \ddots& \ddots & \ddots&\\
			&&-1&0&1\\
			1&&&-1 &0\\
		\end{pmatrix}_{N\times N}, \\[0.5cm]
		&\bbD_2(N,h)=\frac{1}{h}\begin{pmatrix}
			-1 & 1 &&& \\
			& -1 &1&&\\
			&  & \ddots &\ddots & \\
			&& &-1&1\\
			1&&&&-1\\
		\end{pmatrix}_{N\times N}, \quad
		\bbD_3(N,h)=
		\frac{1}{ h^2}\begin{pmatrix}
			-2 & 1 &&&1\\
			1 & -2 &1&&\\
			&  \ddots& \ddots & \ddots&\\
			&&1&-2&1\\
			1&&&1 &-2\\
		\end{pmatrix}_{N\times N}.	
	\end{align*}
	Based on the above definitions, we denote
	\begin{align*}
		&\bbA_1=\bbC(N_x),\quad \bbA_2 = \bbD_1(N_x,h_x), \quad \bbA_3 = \bbD_2(N_x,h_x),\quad \bbA_4 = \bbD_3(N_x,h_x),\\[0.15cm]
		&\bbB_1 = \bbC(N_y),\quad \bbB_2=\bbD_1(N_y,h_y), \quad \bbB_3=\bbD_2(N_y,h_y), \quad \bbB_4=\bbD_3(N_y,h_y).
	\end{align*}
	Then the discrete divergence operator $\nabla_d \cdot (\bullet)$: $[V_{ew};V_{ns}] \rightarrow V_c$ is define as 
	\begin{align*}
		\nabla_d \cdot \bm{U}=\bbA_3 U +  V \bbB_3^T,\quad \forall \bm{U} = [U;V]\in [V_{ew};V_{ns}].
	\end{align*}
	And the discrete gradient operator  $\nabla_D(\bullet)$: $V_c \rightarrow [V_{ew}; V_{ns}]$ is defined as
	\begin{align*}
		\nabla_D P=[-\bbA_3^T P; - P \bbB_3], \quad \forall P\in V_c.
	\end{align*}
	Additionally, the discrete Laplacian operator $\Delta_h (\bullet)$ : $V_{ew}\cup V_{ns} \cup V_c \rightarrow V_{ew}\cup V_{ns} \cup V_c$ is defined as
	\begin{align*}
		\Delta_h W = \bbA_4 W + W\bbB_4^T,\quad \forall W\in V_{ew}\cup V_{ns} \cup V_c. 
	\end{align*}
	For $\bm{U} = [U;V] \in [V_{ew};V_{ns}]$, we assume that $\bm{F}(\bm{U}) \in [V_{ew};V_{ns}]$ satisfies 
	\begin{equation*}
		\big(\bm{F}(\bm{U}), \bm{U}\big)_h \neq 0, \quad \forall \|\bm{U}\|_h \neq 0,
	\end{equation*}
	and then define $\bm{G}(\bm{U}) = [G_1(\bm{U}); G_2(\bm{U})] \in [V_{ew};V_{ns}]$ with elements
	\begin{equation*}
		G_1(\bm{U}) = \begin{cases}
			\frac{U\odot (\mathbb{A}_2 U) - \big((\mathbb{A}_1^T V)\odot (U\mathbb{B}_3)\big)\mathbb{B}_1^T}{(\bm{F}(\bm{U}),\bm{U})_h}, & \|\bm{U}\|_h \neq 0,\\
			\bm{0}, & \text{otherwise},
		\end{cases}
	\end{equation*}
	and
	\begin{equation*}
		G_2(\bm{U}) = \begin{cases}
			\frac{V\odot (V \bbB_2^T) - \bbA_1\big((\bbA_3^T V) \odot (U\bbB_1)\big)}{(\bm{F}(\bm{U}),\bm{U})_h}, & \|\bm{U}\|_h \neq 0,\\
			\bm{0}, & \text{otherwise},
		\end{cases}
	\end{equation*}
	where the notation $\odot$ denotes the Hadamard product of any two matrices. Let 
	\begin{align*}
		\bm{B}(\bm{U},\bm{V})=\big(\bm{F}(\bm{U}),\bm{V}\big)_h \bm{G}(\bm{U})-\big(\bm{G}(\bm{U}),\bm{V}\big)_h \bm{F}(\bm{U}).
	\end{align*}
	We note that $\bm{F}(\bm{U})$, $\bm{G}(\bm{U})$ and $\bm{B}(\bm{U},\bm{V})$ represent the spatial discretization approximations of $\bm{F}(\bm{u})$, $\bm{G}(\bm{u})$ and $\bm{B}(\bm{u},\bm{v})$, respectively. Then we have the following lemme.
	
	\begin{lem}\label{convectproperty3}
		For $\bm{U},~\bm{V}\in [V_{ew};V_{ns}]$ and $P\in V_{c},$ there exist the following identities
		\begin{equation}
			\big(\bm{B}(\bm{U},\bm{V}),\bm{V}\big)_h=0, \quad (\nabla_D P, \bm{U})_h = -( P, \nabla_d \cdot \bm{U})_h.
		\end{equation}
		Furthermore, it holds for $\nabla_d \cdot \bm{U} = \bm{0}$ that
		\begin{equation}\label{property0}
			(\nabla_D P, \bm{U})_h = 0.
		\end{equation}
	\end{lem}
	\begin{proof}
		Through a straightforward calculation, we can deduce
		\begin{align*}
			\big(\bm{B}(\bm{U},\bm{V}),\bm{V}\big)_h=\big(\bm{F}(\bm{U}),\bm{V}\big)_h\cdot \big(\bm{G}(\bm{U}),\bm{V}\big)_h-\big(\bm{G}(\bm{U}),\bm{V}\big)_h\cdot \big(\bm{F}(\bm{U}), \bm{V}\big)_h=0.
		\end{align*}
		For any matrices $A$, $B$, $X$ and $Y$, we have 
		\begin{align}
			&(A X, Y)_h=h_x h_y \mathrm{Tr}((AX)^TY)=h_xh_y\mathrm{Tr}(X^TA^T Y)=(X,A^TY)_h,\label{MP1}\\
			&(XB, Y)_h=(B^TX^T, Y^T)_h =(X^T, BY^T)_h =(X,YB^T)_h\label{MP2}.
		\end{align}
		Using \eqref{MP1} and \eqref{MP2} yields
		\begin{align*}
			(\nabla_D P, \bm{U})_h=&(-\bbA_3^T P,U)_h+(- P \bbB_3,V)_h\\=&(- P,\bbA_3 U)_h+(- P ,V\bbB_3^T)_h\\= &-( P, \nabla_d \cdot \bm{U})_h.
		\end{align*}
		Therefore, we further arrive at \eqref{property0} for the case where $\nabla_d \cdot \bm{U} = \bm{0}$. This completes the proof.
	\end{proof}
	
	\subsection{Fully discrete CN schemes}
	Applying staggered-grid finite difference methods in space to the systems \eqref{CN1scheme} and \eqref{CN2scheme}, respectively, we derive the fully discrete CN1 and CN2 schemes as follows.
	\begin{sch}[CN1]\label{sch:FCN1}
		Given $\bm{U}^0\in [V_{ew};V_{ns}]$, we compute $\bm{U}^{n+1}\in [V_{ew};V_{ns}]$ for $0\leq n \leq N_t-1$ via
		\begin{subequations}
			\label{fullyfirstorderscheme1}
			\begin{numcases}{} 
				\frac{\bm{U}^{n+1}-\bm{U}^n}{\tau}-\nu \Delta_h \bm{U}^{n+\frac{1}{2}}+\bm{B}(\bm{U}^{n},\bm{U}^{n+\frac{1}{2}})+\nabla_D P^{n+\frac{1}{2}}=\mathscr{F}^{n+\frac{1}{2}},\label{fullyfirstorderscheme1a}\\[0.15cm]
				\nabla_d \cdot \bm{U}^{n+\frac{1}{2}}=\bm{0},
			\end{numcases}
		\end{subequations}
		where $\bm{U}^{n+\frac{1}{2}}=(\bm{U}^n+\bm{U}^{n+1})/2$, and $\mathscr{F}^{n+\frac{1}{2}}$ represents the spatial discretization of the external force $\bm{f}^{n+\frac{1}{2}}$.
	\end{sch}
	
	\begin{sch}[CN2]\label{sch:FCN2}
		Given $\bm{U}^0\in [V_{ew};V_{ns}]$, and $\bm{U}^1\in [V_{ew};V_{ns}]$ is obtained using the fully discrete CN1 scheme, we compute $\bm{U}^{n+1}\in [V_{ew};V_{ns}]$ for $1\leq n \leq N_t-1$ via
		\begin{subequations}
			\label{fullyfirstorderscheme2}
			\begin{numcases}{} 
				\frac{\bm{U}^{n+1}-\bm{U}^n}{\tau}-\nu \Delta_h \bm{U}^{n+\frac{1}{2}}+\bm{B}(\overline{\bm{U}}^{n+\frac{1}{2}},\bm{U}^{n+\frac{1}{2}})+\nabla_D P^{n+\frac{1}{2}}=\mathscr{F}^{n+\frac{1}{2}},\label{fullyfirstorderscheme2a}\\[0.15cm]
				\nabla_d \cdot \bm{U}^{n+\frac{1}{2}}=\bm{0},
			\end{numcases}
		\end{subequations}
		where $\overline{\bm{U}}^{n+\frac{1}{2}}=(3\bm{U}^{n}-\bm{U}^{n-1})/2$.
	\end{sch}
	
	\begin{thm}
		In the absence of the external force, both the fully discrete CN1 scheme \eqref{fullyfirstorderscheme1} and the fully discrete CN2 scheme \eqref{fullyfirstorderscheme2} satisfy the following discrete energy dissipation law
		\begin{equation}\label{fullyenergydissipation1}
			\frac{E_h^{n+1}-E_h^n}{\tau}=\nu(\Delta_h \bm{U}^{n+\frac{1}{2}},\bm{U}^{n+\frac{1}{2}})_h, 
		\end{equation}
		where the fully discrete energy at $t = t_n$ is defined as
		\begin{align}
			E_h^n=\frac{1}{2}\| \bm{U}^n \|^2_h.
		\end{align}
	\end{thm}
	
	\begin{proof}
		Similar to the proof of Theorem \ref{thm:CN_EDL}, taking the discrete inner product of \eqref{fullyfirstorderscheme1a} or \eqref{fullyfirstorderscheme2a}  with $\bm{U}^{n+\frac{1}{2}}$ and using Lemma \ref{convectproperty3}, we have
		\begin{align*}
			\frac{E_h^{n+1}-E_h^n}{\tau}=\frac{\| \bm{U}^{n+1}\|^2_h-\| \bm{U}^{n}\|^2_h}{2\tau}=\left(\frac{\bm{U}^{n+1}-\bm{U}^n}{\tau}, \bm{U}^{n+\frac{1}{2}}\right)=\nu(\Delta_h \bm{U}^{n+\frac{1}{2}},\bm{U}^{n+\frac{1}{2}})_h\leq 0.
		\end{align*}
	\end{proof}
	
	\begin{thm}\label{thm:CNunique}
		For $\nu, ~\tau > 0$, under the solvability condition $(P^{n+\frac{1}{2}},1)_h = 0$, both fully discrete CN schemes are uniquely solvable.
	\end{thm}
	
	\begin{proof}
		Here, we only prove the existence and uniqueness of the solution for Scheme \ref{sch:FCN2}. For Scheme \ref{sch:FCN1}, it is similar and thus omitted. For brevity, we denote $\bm{U}:=\bm{U}^{n+\frac{1}{2}},~ \overline{\bm{U}} := \overline{\bm{U}}^{n+\frac{1}{2}},$ etc. Then the linear system \eqref{fullyfirstorderscheme2} can be written as  
		\begin{subequations}\label{Uniquesystem}
			\begin{numcases}{} 
				\frac{2(\bm{U}-\bm{U}^n)}{\tau}-\nu \Delta_h \bm{U}+\bm{B}(\overline{\bm{U}},\bm{U})+
				\nabla_D P=\mathscr{F} , \\[0.15cm]
				\nabla_d \cdot \bm{U}=\bm{0},
			\end{numcases}
		\end{subequations}
		whose homogeneous linear system reads as
		\begin{subequations}\label{Uniquesystem1}
			\begin{numcases}{} 
				\frac{2}{\tau}\bm{U}-\nu \Delta_h \bm{U}+\bm{B}(\overline{\bm{U}},\bm{U})+
				\nabla_D P=\bm{0} , \label{Uniquesystem1a} \\[0.15cm]
				\nabla_d \cdot \bm{U}=\bm{0}.
			\end{numcases}
		\end{subequations}
		Following Theorem 5.1 in Ref.  \cite{GONG2018SIAM}, we need only to prove that the homogeneous linear system \eqref{Uniquesystem1} with $(P,1)_h = 0$ admits only the zero solution. Taking the discrete inner product of \eqref{Uniquesystem1a}  with $\bm{U}$ and using Lemma \ref{convectproperty3}, we can derive
		$$\frac{2}{\tau}\|\bm{U}\|_h^2-\nu (\Delta_h \bm{U}, \bm{U})_h = 0.$$
		Since $\Delta_h$ is a semi-negative definite operator and $\nu, ~\tau > 0$, it follows that
		$$\bm{U}=\bm{0}.$$
		Furthermore, we can derive from \eqref{Uniquesystem1a}  that
		$$\nabla_D P = \bm{0}.$$
		Under the solvability condition $(P,1)_h = 0$, we obtain $$P = \bm{0}.$$
		This completes the proof.
	\end{proof}
	
	In what follows, we present an explicit and efficient implementation of the fully discrete CN2 scheme. The fully discrete CN1 scheme can be solved similarly, and the details are omitted for brevity. We denote
	\begin{align*}
		\alpha=(\bm{F}(\overline{\bm{U}}^{n+\frac{1}{2}}),\bm{U}^{n+\frac{1}{2}})_h,\quad \beta=(\bm{G}(\overline{\bm{U}}^{n+\frac{1}{2}}),\bm{U}^{n+\frac{1}{2}})_h,
	\end{align*}
	and then have
	\begin{align*}
		\bm{B}(\overline{\bm{U}}^{n+\frac{1}{2}},\bm{U}^{n+\frac{1}{2}})=\alpha \bm{G}(\overline{\bm{U}}^{n+\frac{1}{2}})-\beta \bm{F}(\overline{\bm{U}}^{n+\frac{1}{2}}).
	\end{align*}
	Therefore, the linear system \eqref{fullyfirstorderscheme2} can be reformulated as
	\begin{subequations}
		\begin{numcases}{} 
			\frac{2(\bm{U}^{n+\frac{1}{2}}-\bm{U}^n)}{\tau}-\nu \Delta_h \bm{U}^{n+\frac{1}{2}}+\alpha \bm{G}(\overline{\bm{U}}^{n+\frac{1}{2}})-\beta \bm{F}(\overline{\bm{U}}^{n+\frac{1}{2}})+\nabla_D P^{n+\frac{1}{2}}=\mathscr{F}^{n+\frac{1}{2}}, \\[0.15cm]
			\nabla_d \cdot \bm{U}^{n+\frac{1}{2}}=\bm{0}.
		\end{numcases}
	\end{subequations}
	Similar to the semi-discrete system, we decompose the solutions as follows
	\begin{subequations}
		\label{fullydecompose}
		\begin{numcases}{} 
			\bm{U}^{n+\frac{1}{2}}=\alpha \bm{U}^{n+\frac{1}{2}}_1+\beta \bm{U}^{n+\frac{1}{2}}_2+\bm{U}^{n+\frac{1}{2}}_3,\label{fullydecomposea}\\[0.15cm]
			P^{n+\frac{1}{2}}=\alpha P^{n+\frac{1}{2}}_1+\beta P^{n+\frac{1}{2}}_2+P^{n+\frac{1}{2}}_3,\label{fullydecomposeb}
		\end{numcases}
	\end{subequations}
	where $\bm{U}^{n+\frac{1}{2}}_i$ and $P^{n+\frac{1}{2}}_i\ (i=1,2,3)$ satisfy the following three subsystems
	\begin{subequations}
		\label{implesystemfully}
		\begin{numcases}{} 
			\frac{2}{\tau}\bm{U}^{n+\frac{1}{2}}_i-\nu \Delta_h \bm{U}_i^{n+\frac{1}{2}}+\nabla_D P^{n+\frac{1}{2}}_i=\bm{M}^{n}_i, \label{implesystemfullya}\\[0.15cm]
			\nabla_d \cdot \bm{U}^{n+\frac{1}{2}}_i=\bm{0},\label{implesystemfullyb}
		\end{numcases}
	\end{subequations}
	with  
	\begin{align}
		\bm{M}^{n}_1=-\bm{G}(\overline{\bm{U}}^{n+\frac{1}{2}}), \quad \bm{M}^{n}_2=\bm{F}(\overline{\bm{U}}^{n+\frac{1}{2}}), \quad \bm{M}^{n}_3=\frac{2}{\tau}\bm{U}^n+\mathscr{F}^{n+\frac{1}{2}}.
	\end{align}
	Applying the discrete divergence operator $\nabla_d \cdot (\bullet)$ to \eqref{implesystemfullya} and utilizing \eqref{implesystemfullyb}, we obtain
	\begin{align}\label{implesystemfullyP}
		\Delta_h P^{n+\frac{1}{2}}_i=\nabla_d \cdot \bm{M}^{n}_i.
	\end{align}
	It is noted that the system \eqref{implesystemfullyP} under the condition $(P^{n+\frac{1}{2}}_i,1)_h = 0$ is uniquely solvable and can be efficiently solved using the FFT algorithm \cite{JUJSC2015,WangJSC2020}. Once $P^{n+\frac{1}{2}}_i$ is determined, we can compute $\bm{U}^{n+\frac{1}{2}}_i$ via 
	\begin{align}\label{explicitimple}
		\Big(\frac{2}{\tau}-\nu \Delta_h \Big) \bm{U}^{n+\frac{1}{2}}_i = \bm{M}^{n}_i-\nabla_DP^{n+\frac{1}{2}}_i,
	\end{align}
	which is also uniquely solvable and can be efficiently solved using the FFT algorithm. Once $\bm{U}_i^{n+\frac{1}{2}}~(i=1,2,3)$ are obtained, the coefficients $\alpha$ and $\beta$ can be computed by solving the following $2 \times 2$ linear system
	\begin{align}\label{MatrixequationFULLY}
		A x =	b,
	\end{align}
	where 
	\begin{align*}
		A = \begin{pmatrix}
			1-(\bm{F}(\overline{\bm{U}}^{n+\frac{1}{2}}),\bm{U}^{n+\frac{1}{2}}_1)_h & -(\bm{F}(\overline{\bm{U}}^{n+\frac{1}{2}}),\bm{U}^{n+\frac{1}{2}}_2)_h\\
			-(\bm{G}(\overline{\bm{U}}^{n+\frac{1}{2}}),\bm{U}^{n+\frac{1}{2}}_1)_h &  1-(\bm{G}(\overline{\bm{U}}^{n+\frac{1}{2}}),\bm{U}^{n+\frac{1}{2}}_2)_h
		\end{pmatrix},\quad
		x = \begin{pmatrix}
			\alpha\\
			\beta
		\end{pmatrix},\quad b = \begin{pmatrix}
			(\bm{F}(\overline{\bm{U}}^{n+\frac{1}{2}}),\bm{U}^{n+\frac{1}{2}}_3)_h\\
			(\bm{G}(\overline{\bm{U}}^{n+\frac{1}{2}}),\bm{U}^{n+\frac{1}{2}}_3)_h
		\end{pmatrix}.
	\end{align*}
	After solving $\alpha$ and $\beta$, $\bm{U}^{n+\frac{1}{2}}$ and $P^{n+\frac{1}{2}}$ can be updated by utilizing \eqref{fullydecomposea} and \eqref{fullydecomposeb}, respectively. Finally, we compute 
	\begin{align}\label{UpdatefinalCNFULLY}
		\bm{U}^{n+1} = 2\bm{U}^{n+\frac{1}{2}} - \bm{U}^n.
	\end{align}
	The implementation of our fully discrete CN2 scheme is summarized in Algorithm \ref{power2}.
	\begin{algorithm}[H]
		\renewcommand{\algorithmicrequire}{\textbf{Date:}}
		\renewcommand{\algorithmicensure}{\textbf{Result:}}
		\caption{Pseudocode outlining the implementation of Scheme \ref{sch:FCN2} for each time step}
		\label{power2}
		\begin{algorithmic}[1] % 控制是否有序号
			\REQUIRE  Input $\bm{U}^{n-1}$, $\bm{U}^{n}$; % input 的内容		
			\STATE Calculate $P_i^{n+\frac{1}{2}}$ by solving \eqref{implesystemfullyP};
			\STATE Calculate $\bm{U}_i^{n+\frac{1}{2}}$ by solving \eqref{explicitimple};
			\STATE Calculate $\alpha$ and $\beta$ by solving \eqref{MatrixequationFULLY};
			\STATE Update $\bm{U}^{n+\frac{1}{2}}$ and $P^{n+\frac{1}{2}}$ using \eqref{fullydecomposea} and \eqref{fullydecomposeb}, respectively;
			\STATE Update $\bm{U}^{n+1}$ using \eqref{UpdatefinalCNFULLY};
			\ENSURE Output $\bm{U}^{n+1}$. % output 的内容
		\end{algorithmic}
	\end{algorithm}
	
	\subsection{Fully discrete BDF schemes}
	
	Applying staggered-grid finite difference methods in space to the systems \eqref{BDF1scheme} and \eqref{BDF2scheme}, respectively, we derive the fully discrete BDF1 and BDF2 schemes as follows.
	\begin{sch}[BDF1]\label{sch:FBDF1}
		Given $\bm{U}^0 \in [V_{ew};V_{ns}]$, we compute $\bm{U}^{n+1} \in [V_{ew};V_{ns}]$ for $0\leq n \leq N_t-1$ via
		\begin{subequations}
			\label{fullyfirstorderscheme4}
			\begin{numcases}{} 
				\frac{\bm{U}^{n+1}-\bm{U}^n}{\tau}-\nu \Delta_h \bm{U}^{n+1}+\bm{B}(\bm{U}^n,\bm{U}^{n+1})+\nabla_D P^{n+1}=\mathscr{F}^{n+1},\label{fullyfirstorderscheme4a}\\[0.15cm]
				\nabla_d \cdot \bm{U}^{n+1}=\bm{0},
			\end{numcases}
		\end{subequations}
		where $\mathscr{F}^{n+1}$ represents the spatial discretization of the external force $\bm{f}^{n+1}$.
	\end{sch}
	
	%To construct the second-order fully discrete BDF scheme, we assume  $\bm{U}^{n-1}$ and $\bm{U}^n$ are given and denote $\overline{\bm{U}}^{n+1}=2\bm{U}^{n}-\bm{U}^{n-1}$ as the second-order extrapolation for approximation $\bm{U}(t_{n+1})$. Then we develop the following second-order fully discrete linear CN scheme.
	
	\begin{sch}[BDF2]\label{sch:FBDF2}
		Given $\bm{U}^0 \in [V_{ew};V_{ns}]$, and $\bm{U}^1 \in [V_{ew};V_{ns}]$ is obtained using the fully discrete BDF1 scheme, we compute $\bm{U}^{n+1} \in [V_{ew};V_{ns}]$ for $1\leq n \leq N_t-1$ via
		\begin{subequations}
			\label{fullyfirstorderscheme3}
			\begin{numcases}{} 
				\frac{3\bm{U}^{n+1}-4\bm{U}^n+\bm{U}^{n-1}}{2\tau}-\nu \Delta_h \bm{U}^{n+1}+\bm{B}(\overline{\bm{U}}^{n+1},\bm{U}^{n+1})+\nabla_D P^{n+1}=\mathscr{F}^{n+1},\label{fullyfirstorderscheme3a}\\[0.15cm]
				\nabla_d \cdot \bm{U}^{n+1}=\bm{0},
			\end{numcases}
		\end{subequations}
		where $\overline{\bm{U}}^{n+1}=2\bm{U}^n-\bm{U}^{n-1}$.
	\end{sch}

	\begin{thm}\label{fullyfirstorderenergydissipation}
		In the absence of the external force, the fully discrete BDF1 scheme \eqref{fullyfirstorderscheme4} preserves the discrete energy dissipation law 
		\begin{equation}
			\frac{E_h^{n+1}-E_h^n}{\tau}=\nu(\Delta_h \bm{U}^{n+1},\bm{U}^{n+1})_h-\frac{1}{2\tau}\| \bm{U}^{n+1}-\bm{U}^n\|^2_h,
		\end{equation}
		where the fully discrete energy of BDF1 scheme  at $t = t_n$ is defined as
		\begin{align}
			E_h^n=\frac{1}{2}\| \bm{U}^n \|^2_h.
		\end{align}
		Additionally, the fully discrete BDF2 scheme \eqref{fullyfirstorderscheme3} satisfies the discrete energy dissipation law 
		\begin{equation}
			\frac{\wh{E}_h^{n+1}-\wh{E}_h^n}{\tau}=\nu(\Delta_h \bm{U}^{n+1},\bm{U}^{n+1})_h-\frac{1}{4\tau}\| \bm{U}^{n+1}-2\bm{U}^n+\bm{U}^{n-1}\|^2_h,
		\end{equation}
		where the fully discrete energy of BDF2 scheme at $t = t_n$ is defined as
		\begin{align}
			\wh{E}_h^n=\frac{1}{4}\left(\| \bm{U}^n \|^2_h+\|2\bm{U}^n-\bm{U}^{n-1}\|^2_h\right).
		\end{align}
	\end{thm}
	
	\begin{proof}
		Here, we present the proof of the discrete energy dissipation law for the fully discrete BDF2 scheme. The proof for the fully discrete BDF1 scheme follows analogously, and thus the details are omitted for brevity.
		
		Taking the discrete inner product of \eqref{fullyfirstorderscheme3a} with $\bm{U}^{n+1}$ and using Lemma \ref{convectproperty3}, one can arrive at
		\begin{align*}
			\left(\frac{3\bm{U}^{n+1}-4\bm{U}^{n}+\bm{U}^{n-1}}{2\tau},\bm{U}^{n+1}\right)_h = \nu(\Delta_h \bm{U}^{n+1},\bm{U}^{n+1})_h.
		\end{align*}
		By applying the following identity
		\begin{align*}
			\left(3\bm{U}^{n+1}-4\bm{U}^{n}+\bm{U}^{n-1},2\bm{U}^{n+1}\right)_h=&\| \bm{U}^{n+1} \|^2_h-\| \bm{U}^{n}\|^2_h+ \| 2\bm{U}^{n+1} -\bm{U}^{n}\|^2_h- \| 2\bm{U}^{n} -\bm{U}^{n-1}\|^2_h\\ &+ \| \bm{U}^{n+1} -2\bm{U}^{n}+\bm{U}^{n-1}\|^2_h,
		\end{align*}
		we obtain
		\begin{align*}
			\frac{\wh{E}_h^{n+1}-\wh{E}_h^n}{\tau}=\nu(\Delta_h \bm{U}^{n+1},\bm{U}^{n+1})_h-\frac{1}{4\tau}\| \bm{U}^{n+1}-2\bm{U}^n+\bm{U}^{n-1}\|^2_h\leq 0.
		\end{align*}
	\end{proof}
	
	\begin{thm}\label{Uniquethm2}
		For $\nu,~ \tau > 0$,  under the solvability condition $(P^{n+1},1)_h=0$, both fully discrete BDF schemes are uniquely solvable.
	\end{thm}
	
	\begin{proof}
		We also only prove the existence and uniqueness of the solution for Scheme \ref{sch:FBDF2} and omit the similar proof of the version of Scheme  \ref{sch:FBDF1}. For brevity, we denote $\bm{U}:=\bm{U}^{n+1},~ \overline{\bm{U}} := \overline{\bm{U}}^{n+1},$ etc. Then the linear system \eqref{fullyfirstorderscheme3} can be written as
		\begin{subequations}\label{UniquesystemBDF}
			\begin{numcases}{} 
				\frac{3\bm{U}-4\bm{U}^n+\bm{U}^{n-1}}{2\tau}-\nu \Delta_h \bm{U}+\bm{B}(\overline{\bm{U}},\bm{U})+
				\nabla_D P=\mathscr{F} , \\[0.15cm]
				\nabla_d \cdot \bm{U}=\bm{0},
			\end{numcases}
		\end{subequations}
		whose homogeneous linear system reads as 
		\begin{subequations}\label{Uniquesystem2}
			\begin{numcases}{} 
				\frac{3}{2\tau}\bm{U}-\nu \Delta_h \bm{U}+\bm{B}(\overline{\bm{U}},\bm{U})+
				\nabla_D P=\bm{0} , \label{Uniquesystem1BDFa} \\[0.15cm]
				\nabla_d \cdot \bm{U}=\bm{0}.
			\end{numcases}
		\end{subequations}
		Similar to the proof of Theorem \ref{thm:CNunique}, we need only to prove that the homogeneous linear system \eqref{Uniquesystem2} with $(P,1)_h = 0$ admits only the zero solution. Taking the discrete inner product of \eqref{Uniquesystem1BDFa} with $\bm{U}$ and using Lemma \ref{convectproperty3}, we have
		\begin{align}
			\frac{3}{2\tau}\|\bm{U}\|^2_h - \nu(\Delta_h \bm{U},\bm{U})_h=0.
		\end{align}
		Since $\Delta_h$ is a semi-negative definite operator and $\nu, ~\tau > 0$, it follows that
		$$\bm{U}=\bm{0}.$$
		Furthermore, we can derive from \eqref{Uniquesystem1BDFa}  that
		$$\nabla_D P = \bm{0}.$$
		Noticing the solvability condition $(P,1)_h = 0$, we obtain $$P = \bm{0}.$$
		This completes the proof.
	\end{proof}
	
	In what follows, we present an explicit and effcient implementation of the fully discrete BDF2 scheme. The fully discrete BDF1 scheme can be solved similarly, and the details are omitted for brevity. We denote
	\begin{align*}
		\alpha=(\bm{F}(\overline{\bm{U}}^{n+1}),\bm{U}^{n+1})_h,\quad \beta=(\bm{G}(\overline{\bm{U}}^{n+1}),\bm{U}^{n+1})_h,
	\end{align*}
	and then have
	\begin{align*}
		\bm{B}(\overline{\bm{U}}^{n+1},\bm{U}^{n+1})=\alpha \bm{G}(\overline{\bm{U}}^{n+1})-\beta \bm{F}(\overline{\bm{U}}^{n+1}).
	\end{align*}
	Therefore, the linear system \eqref{fullyfirstorderscheme3} can be reformulated as
	\begin{subequations}
		\begin{numcases}{} 
			\frac{3\bm{U}^{n+1}-4\bm{U}^n+\bm{U}^{n-1}}{2\tau}-\nu \Delta_h \bm{U}^{n+1}+\alpha \bm{G}(\overline{\bm{U}}^{n+1})-\beta \bm{F}(\overline{\bm{U}}^{n+1})+\nabla_D P^{n+1}=\mathscr{F}^{n+1}, \\[0.15cm]
			\nabla_d \cdot \bm{U}^{n+1}=\bm{0}.
		\end{numcases}
	\end{subequations}
	Similar to the semi-discrete system, we decompose the solutions as follows
	\begin{subequations}
		\label{fullydecomposeBDF}
		\begin{numcases}{} 
			\bm{U}^{n+1}=\alpha \bm{U}^{n+1}_1+\beta \bm{U}^{n+1}_2+\bm{U}^{n+1}_3,\label{fullydecomposeBDFa}\\[0.15cm]
			P^{n+1}=\alpha P^{n+1}_1+\beta P^{n+1}_2+P^{n+1}_3,\label{fullydecomposeBDFb}
		\end{numcases}
	\end{subequations}
	where $\bm{U}^{n+1}_i$ and $P^{n+1}_i\ (i=1,2,3)$ satisfy the following three subsystems
	\begin{subequations}
		\label{implesystemfullyBDF}
		\begin{numcases}{} 
			\frac{3}{2\tau}\bm{U}^{n+1}_i-\nu \Delta_h \bm{U}^{n+1}_i+\nabla_D P^{n+1}_i=\bm{M}^{n}_i, \label{implesystemfullyBDFa}\\[0.15cm]
			\nabla_d \cdot \bm{U}^{n+1}_i=\bm{0},\label{implesystemfullyBDFb}
		\end{numcases}
	\end{subequations}
	with  
	\begin{align}
		\bm{M}^{n}_1=-\bm{G}(\overline{\bm{U}}^{n+1}), \quad \bm{M}^{n}_2=\bm{F}(\overline{\bm{U}}^{n+1}), \quad \bm{M}^{n}_3=\frac{1}{2\tau} \left(4\bm{U}^n-\bm{U}^{n-1}\right)+\mathscr{F}^{n+1}.
	\end{align}
	Applying the discrete divergence operator $\nabla_d \cdot (\bullet)$ to \eqref{implesystemfullyBDFa} and utilizing \eqref{implesystemfullyBDFb}, we obtain
	\begin{align}
		\label{implesystemfullyBDFP}
		\Delta_h P^{n+1}_i=\nabla_d \cdot \bm{M}^{n}_i.
	\end{align}
	It is noted that the system \eqref{implesystemfullyBDFP} under the condition $(P^{n+1}_i,1)_h=0$ is unique solvable and can be efficiently solved using the FFT algorithm. Once $P_i^{n+1}$ is determined, we can compute $\bm{U}^{n+1}_i$ via
	\begin{align}
		\label{explicitimpleBDF}
		\Big(\frac{3}{2\tau}-\nu \Delta_h \Big)\bm{U}^{n+1}_i =\bm{M}^{n}_i-\nabla_DP^{n+1}_i,
	\end{align}
	which is also uniquely solvable and can be efficiently solved using the FFT algorithm. Once $\bm{U}^{n+1}_i~(i=1,2,3)$ are obtained, the coefficients $\alpha$ and $\beta$ can be computed by solving the following $2 \times 2$ linear system
	\begin{align}\label{MatrixequationFULLYBDF}
		A x =	b,
	\end{align}
	where 
	\begin{align*}
		A = \begin{pmatrix}
			1-(\bm{F}(\overline{\bm{U}}^{n+1}),\bm{U}^{n+1}_1)_h & -(\bm{F}(\overline{\bm{U}}^{n+1}),\bm{U}^{n+1}_2)_h\\
			-(\bm{G}(\overline{\bm{U}}^{n+1}),\bm{U}^{n+1}_1)_h &  1-(\bm{G}(\overline{\bm{U}}^{n+1}),\bm{U}^{n+1}_2)_h
		\end{pmatrix},\quad
		x = \begin{pmatrix}
			\alpha\\
			\beta
		\end{pmatrix},\quad b = \begin{pmatrix}
			(\bm{F}(\overline{\bm{U}}^{n+1}),\bm{U}^{n+1}_3)_h\\
			(\bm{G}(\overline{\bm{U}}^{n+1}),\bm{U}^{n+1}_3)_h
		\end{pmatrix}.
	\end{align*}
	After solving $\alpha$ and $\beta$, $\bm{U}^{n+1}$ and $P^{n+1}$ can be updated by utilizing \eqref{fullydecomposeBDFa} and \eqref{fullydecomposeBDFb}, respectively.
	Algorithm \ref{power4} is proposed to summarize the implementation of our fully discrete BDF2 scheme.
	\begin{algorithm}[H]
		\renewcommand{\algorithmicrequire}{\textbf{Date:}}
		\renewcommand{\algorithmicensure}{\textbf{Result:}}
		\caption{Pseudocode outlining the implementation of Scheme \ref{sch:FBDF2} for each time step}
		\label{power4}
		\begin{algorithmic}[1] % 控制是否有序号
			\REQUIRE  Input $\bm{U}^{n-1},~\bm{U}^n$; % input 的内容		
			\STATE Calculate $P_i^{n+1}$ by solving  \eqref{implesystemfullyBDFP};
			\STATE Calculate $\bm{U}_i^{n+1}$ by solving \eqref{explicitimpleBDF};
			\STATE Calculate $\alpha$ and $\beta$ by solving  \eqref{MatrixequationFULLYBDF};
			\STATE Update $\bm{U}^{n+1}$ and $P^{n+1}$ using \eqref{fullydecomposeBDFa} and \eqref{fullydecomposeBDFb}, respectively;
			\ENSURE Output $\bm{U}^{n+1}$. % output 的内容
		\end{algorithmic}
	\end{algorithm}
	
	\begin{rem}
		It is noteworthy that all the proposed fully discrete schemes can be straightforwardly extended to accommodate Dirichlet boundary conditions. Specifically, under homogeneous Dirichlet boundary conditions, these schemes still preserve the corresponding energy dissipation law. For more general Dirichlet boundary conditions, they remain uniquely solvable. In such cases, the generalized Stokes systems can be efficiently solved by employing the preconditioned iterative solvers described in Ref. \cite{doan4938899dynamically}.
	\end{rem}

	\section{Numerical experiments}
	\label{Numerical experiments}
	
	In this section, we present a series of benchmark numerical examples to measure the accuracy and efficiency of our proposed second-order schemes, in comparison with existing methods. The test problems include the accuracy verification, the classical Taylor-Green vortex problem, the lid-driven cavity flow and the Kelvin-Helmholtz instability. As both CN2 and BDF2 schemes demonstrate the same numerical results in the numerical examples except the accuracy verification, we report only the results of the CN2-based method. For simplicity, a uniform spatial grid is employed with $h_x = h_y = h$ for two-dimensional problems and $h_x = h_y = h_z = h$ for three-dimensional ones. The Reynolds number is defined as $Re = 1/\nu$. 
	
	\subsection{Accuracy test}\label{Accuracy test}
	We set the spatial domain $\Omega = [0, 1]^2$ and the final time $T=1$. The external force $\bm{f}$ is computed according to the following analytic solution \cite{SHEN2021SIAM} , which satisfies periodic boundary conditions:
	\begin{align}
		\left\{  
		\begin{array}{lr}  
			u(x,y,t)= \exp(t)\sin^2 (\pi x) \sin (2\pi y) , \\[0.15cm]
			v(x,y,t)=-\exp(t)\sin (2\pi x) \sin^2(\pi  y ), \\[0.15cm]
			p(x,y,t)=\exp(t)\sin(2 \pi x)\sin(2 \pi y).
		\end{array}  
		\right. 
	\end{align}
	The function $\bm{F}(\bm{u})$ is chosen as $\bm{F}(\bm{u})=\bm{u}$. We first consider a relative low Reynolds number $Re=10$. To study the convergence of the proposed schemes, we vary both the spatial mesh size and time step size with $h = \frac{\tau}{10}$ and $\tau \in \left\{\frac{1}{20},~\frac{1}{40},~\frac{1}{80},~\frac{1}{160}\right\}$. For CN2-based schemes, we report the $L^{\infty}$ errors of pressure at time $t=(N_t-0.5)\tau$. For other schemes or variables, the errors are recorded at finial time. Figure~\ref{ERROR} compares the $L^{\infty}$ errors of four methods: our proposed CN2 and BDF2 schemes, the CN2-based dynamically regularized Lagrange multiplier (DRLM) method from \cite{doan4938899dynamically} (using their recommended parameter $\theta=100$), and the CN2-based zero-energy contribution (ZEC) scheme \cite{YANG2020WILEY}. It is clear that all the schemes satisfy the second-order convergence rate, while our CN2 scheme enjoys the smallest $L^{\infty}$ errors for both velocity $\bu$ and pressure $p$.
	
	Next, we increase the Reynolds number to $Re=1000$. Due to the explicit treatment of the convection term which requires some CFL condition, we use $h=4\tau$ and $\tau\in\left\{\frac{1}{400},~\frac{1}{800},~\frac{1}{1600},~\frac{1}{3200}\right\}$. The corresponding errors of the velocity and pressure are recorded in Table \ref{timeaccuracyU} and Table \ref{timeaccuracyP}, respectively. To be specific, all the schemes successfully reach the second-order convergence rate. Additionally, because of the predominance of the spatial errors, the accuracy of various schemes differs little from each other.

		\begin{figure}[ht]
			\centering
			\subfigure[$L^{\infty}$ errors of velocity.]{
				\includegraphics[width=3in]{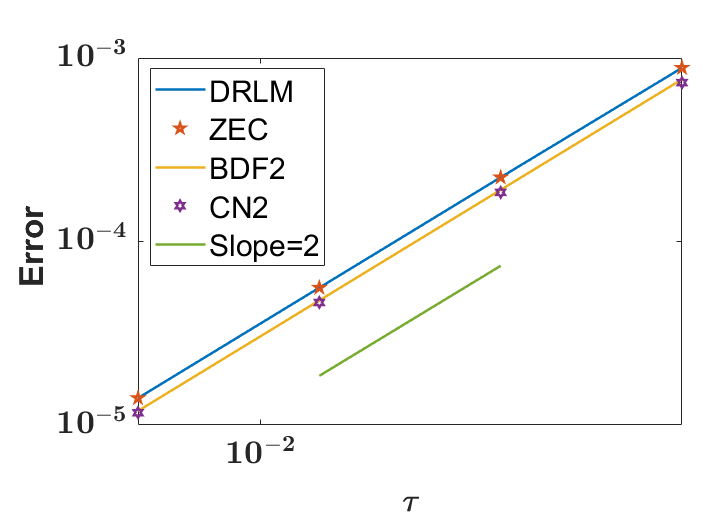}
			}
			\subfigure[$L^{\infty}$ errors of pressure.]{
				\includegraphics[width=3in]{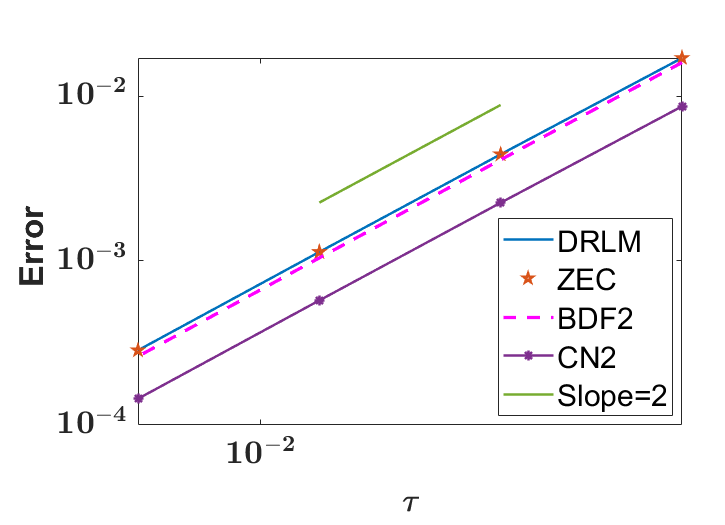}
			}
			\caption{Errors of the velocity and pressure obtained by various schemes with $Re=10$.}\label{ERROR}
		\end{figure} 
		\begin{table}[H]
			\centering
			\caption{\large Convergence rate for velocity field $\bm{u}$ with $Re=1000$.}
			\label{timeaccuracyU}
			\scalebox{0.9}{
				\begin{tabular}{c|ccccc}% 通过添加 | 来表示是否需要绘制竖线
					\hline  % 在表格最上方绘制横线
					& &$\tau=1/400$&$\tau/2$&$\tau/4$&$\tau/8$\\
					\hline
					CN2   &$L^{\infty}$-errors  & 2.0340e-03&5.0660e-04& 1.2630e-04 &3.1540e-05    \\
					&Rate   &  -   &2.00& 2.00&2.00\\
					BDF2 & $L^{\infty}$-errors  &  2.0350e-03 &5.0670e-04&1.2640e-04 &3.1550e-05   \\
					&Rate   &  -   &2.00& 2.00&2.00\\
					DRLM \cite{doan4938899dynamically} & $L^{\infty}$-errors  &  2.0320e-03 &5.0600e-04&1.2620e-04 &3.1510e-05  \\
					&Rate   &  -   &2.00& 2.00&2.00\\
					ZEC \cite{YANG2020WILEY}& $L^{\infty}$-errors  &  2.0320e-03 &5.0600e-04&1.2620e-04 &3.1510e-05  \\
					&Rate   &  -   &2.00& 2.00&2.00\\
					\hline  %在第一行和第二行之间绘制横线
				\end{tabular}
			}
		\end{table}
		
		\begin{table}[H]
			\centering
			\caption{\large Convergence rate for pressure $p$ with $Re=1000$.}
			\label{timeaccuracyP}
			\scalebox{0.9}{
				\begin{tabular}{c|ccccc}% 通过添加 | 来表示是否需要绘制竖线
					\hline  % 在表格最上方绘制横线
					& &$\tau=1/400$&$\tau/2$&$\tau/4$&$\tau/8$\\
					\hline
					CN2   &$L^{\infty}$-errors  & 7.1890e-03&  1.8000e-03  &4.5030e-04&1.1260e-04 \\
					&Rate   &  -   &1.99&1.99 &2.00\\
					BDF2 & $L^{\infty}$-errors  & 7.1960e-03& 1.8000e-03 &4.4990e-04  &1.1250e-04\\
					&Rate   &  -   & 2.00&2.00 &2.00\\
					DRLM \cite{doan4938899dynamically}& $L^{\infty}$-errors  &  7.2110e-03  &1.8060e-03&4.5170e-04 &1.1300e-04   \\
					&Rate   &  -   &1.99& 1.99&2.00\\
					ZEC \cite{YANG2020WILEY} & $L^{\infty}$-errors  &  7.2110e-03  &1.8060e-03&4.5170e-04 &1.1300e-04   \\
					&Rate   &  -   &1.99& 1.99&2.00\\
					\hline  %在第一行和第二行之间绘制横线
				\end{tabular}
			}
		\end{table}
		\subsection{Energy dissipation test}
		In this example, our primary objective is to measure the the energy dissipation property of our developed schemes.  We consider the Taylor-Green vortex problem in the rectangular domain $\Omega = [0, 1]^2$ under periodic boundary conditions. The exact solution to this problem is given as follows:
		\begin{align}
			\left\{  
			\begin{array}{lr}  
				u(x,y,t)=\exp(-8 \pi^2 \nu t) \sin (2\pi x) \cos (2\pi y),  \\[0.15cm]
				v(x,y,t)= -\exp(-8 \pi^2 \nu t)\cos (2\pi x) \sin (2\pi y), \\[0.15cm]
				p(x,y,t)= \frac{1}{4}\exp(-16 \pi^2 \nu t)[\cos(4\pi x)+\cos(4\pi y)].
			\end{array}  
			\right.  
		\end{align}
		The exact expression of kinetic energy is given by
		\begin{align}
			E(t)=\frac{1}{4}\exp(-16 \pi^2 \nu t).
		\end{align}
		\begin{figure}[H]
			\centering
			\subfigure[$Re=1000$, $\tau=\frac{1}{64}$.]{
				\includegraphics[width=2.5in]{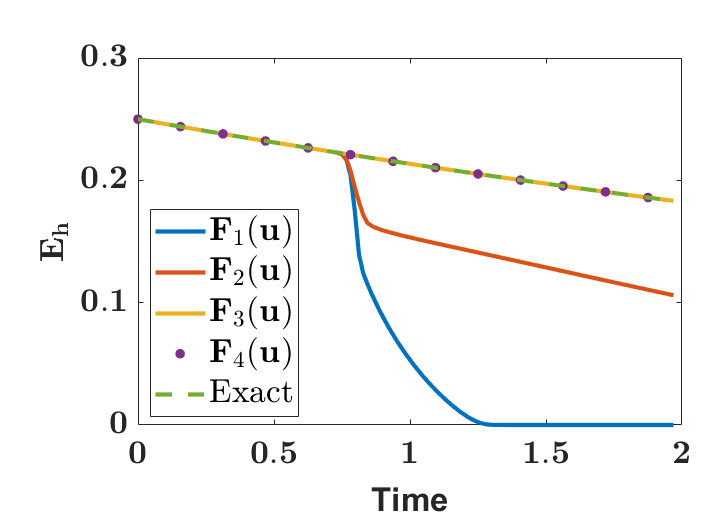}
			}
			\subfigure[$Re=1000$, $\tau=\frac{1}{128}$.]{
				\includegraphics[width=2.5in]{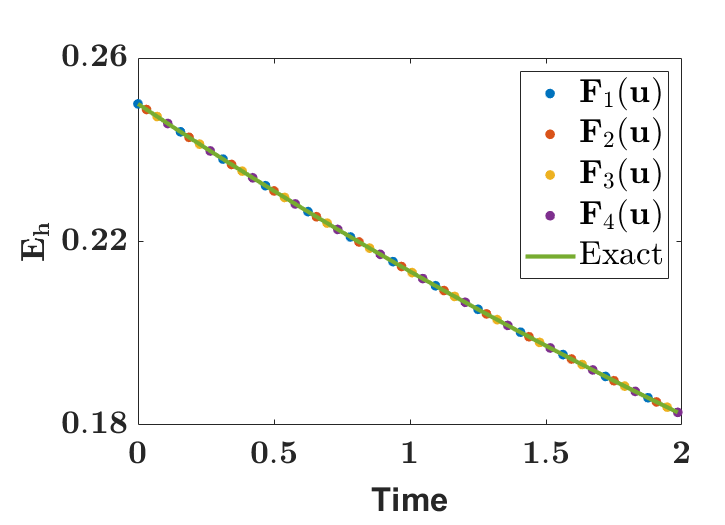}
			}\\
			\subfigure[$Re=10000$, $\tau=\frac{1}{256}$.]{
				\includegraphics[width=2.5in]{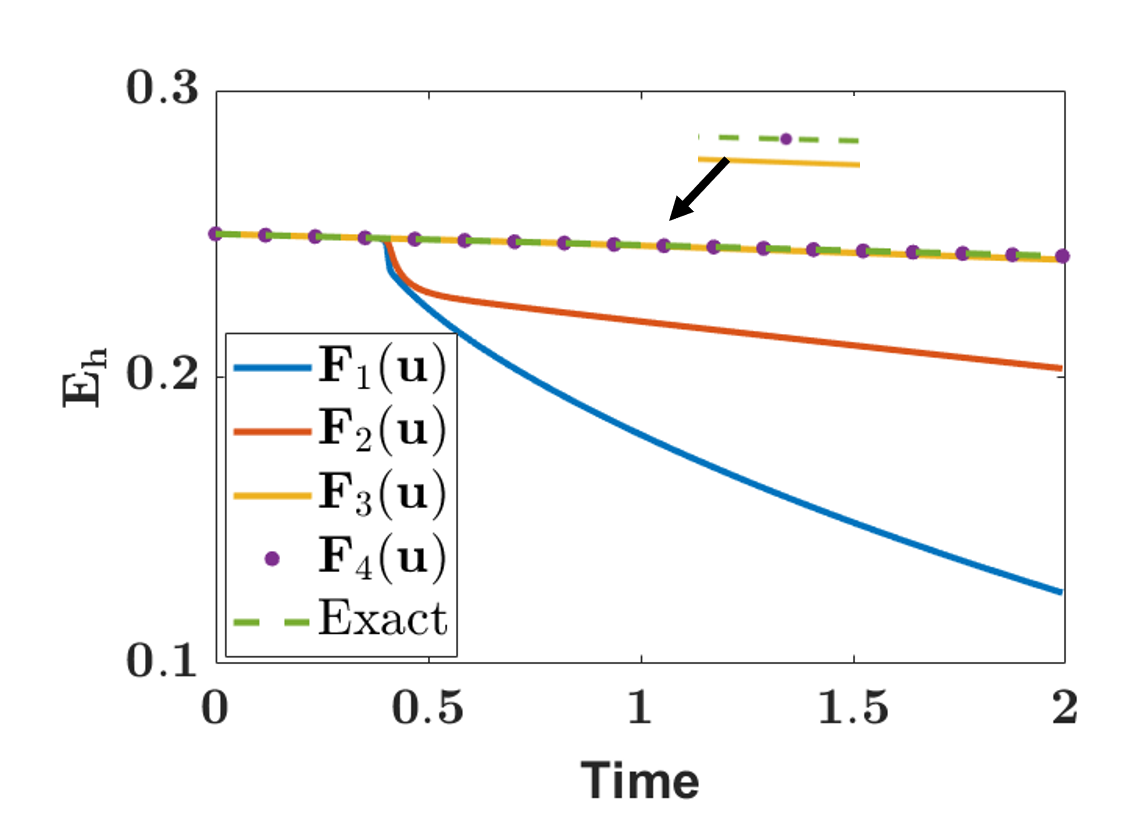}
			}
			\subfigure[$Re=10000$, $\tau=\frac{1}{512}$.]{
				\includegraphics[width=2.5in]{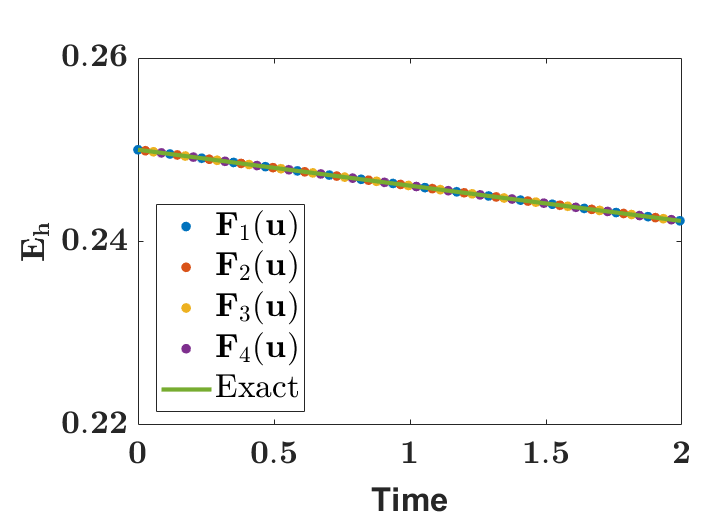}
			}
			\caption{Energy evolution for CN2 scheme with different forms of $\bm{F}(\bm{u})$. Mesh size is $h=\frac{1}{128}$ for $Re=1000$, and $h=\frac{1}{256}$ for $Re=10000$ with different time steps.}\label{CPFU}
		\end{figure} 
		
		To measure the numerical influence of various formulations of $\bm{F} (\bm{u})$ on the energy dissipation performance, we consider $\bm{F} (\bm{u}) \in \left\{ \bm{F}_1(\bm{u})= \bm{u},~ \bm{F}_2(\bm{u})=\bm{u}^3,~ \bm{F}_3(\bm{u}),~ \bm{F}_4(\bm{u})\right\}$, where 
		\begin{align*}
			\bm{F}_3(\bm{u})=	\left\{  
			\begin{array}{lr}  
				\bu,  \quad |\bu|<10^{-10},\\[0.15cm]
				\frac{1}{\bu}, \quad \text{otherwise},
			\end{array}  
			\right.  \quad \bm{F}_4(\bm{u})=	\left\{  
			\begin{array}{lr}  
				\bu,  \quad |\bu|<10^{-10},\\[0.15cm]
				\frac{1}{\bu^3}, \quad  \text{otherwise}.
			\end{array}  
			\right.
		\end{align*}
		For Reynolds number $Re=1000$ and $Re=10000$, we consider the uniform mesh size $h=\frac{1}{128}$ and $h=\frac{1}{256}$, respectively. 
		The energy evolution obtained by using various forms of $\bm{F}(\bm{u})$ is generated in Figure~\ref{CPFU}. We observe that, under different time step sizes, all computed curves exhibit a monotonic decay for both two Reynolds numbers, confirming that our algorithm preserves the energy dissipation property unconditionally. However, it can be seen that for relatively coarse time step size, our energy evolution curves obtained by using $\bm{F}_1(\bm{u})$ or $\bm{F}_2(\bm{u})$, deviate significantly from the exact energy. In contrast, the energy curves obtained by both $\bm{F}_3(\bm{u})$ and $\bm{F}_4(\bm{u})$ closely match the exact solution for both time step sizes. Furthermore, the energy curve given by using $\bm{F}_4(\bm{u})$ is more accurate than the one obtained with $\bm{F}_3(\bm{u})$ under $Re=10000$. To further investigate the observed phenomenon, we plot the temporal evolution of the quantity $|(\overline{\bm{u}}^{n+\frac{1}{2}} \cdot \nabla\overline{\bm{u}}^{n+\frac{1}{2}},\bm{u}^{n+\frac{1}{2}})_h|$ in Figure~\ref{innerCP} for various $\bm{F}(\bm{u})$, which is equal to zero in the continuous level. It is observed that when the time step is sufficiently small, quantities computed by different $\bm{F}(\bm{u})$ are relatively small. However, as the time step increases, the quantities produced by CN2 scheme with $\bm{F}_1(\bm{u})$ or $\bm{F}_2(\bm{u})$ exhibit remarkable deviations from zero, compared to those obtained with $\bm{F}_4(\bm{u})$. This suggests that the proper choice of the stabilization term $\bm{F}(\bm{u})$ can effectively suppress the pathological growth of the quantity $|(\overline{\bm{u}}^{n+\frac{1}{2}} \cdot \nabla\overline{\bm{u}}^{n+\frac{1}{2}},\bm{u}^{n+\frac{1}{2}})_h|$, thereby mitigating the unphysical oscillations in the energy evolution. All the discoveries inspire us to use $\bm{F}_4(\bm{u})$ while carrying out numerical simulations under large time steps in this example.

		\begin{figure}[H]
			\centering
			\subfigure[$Re=1000$, $\tau=\frac{1}{64},~h=\frac{1}{128}$.]{
				\includegraphics[width=2.6in]{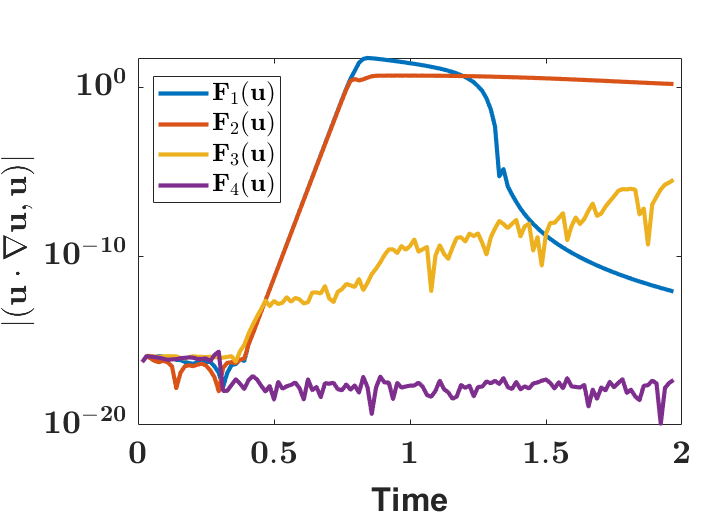}
			}
			\subfigure[$Re=1000$, $\tau=\frac{1}{128},~h=\frac{1}{128}$.]{
				\includegraphics[width=2.6in]{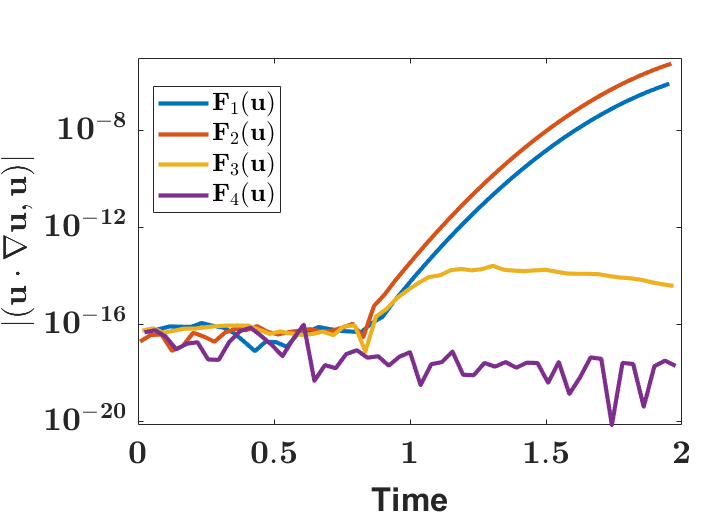}
			}\\
			\subfigure[$Re=10000$, $\tau=\frac{1}{256},~h=\frac{1}{256}$.]{
				\includegraphics[width=2.6in]{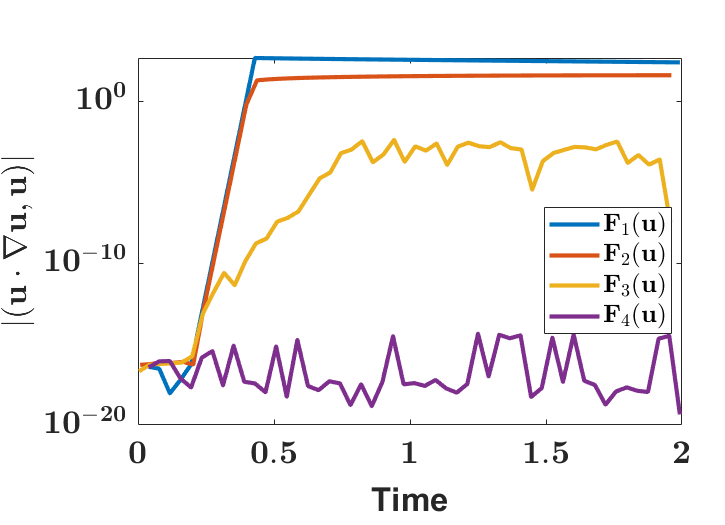}
			}
			\subfigure[$Re=10000$, $\tau=\frac{1}{512},~h=\frac{1}{256}$.]{
				\includegraphics[width=2.6in]{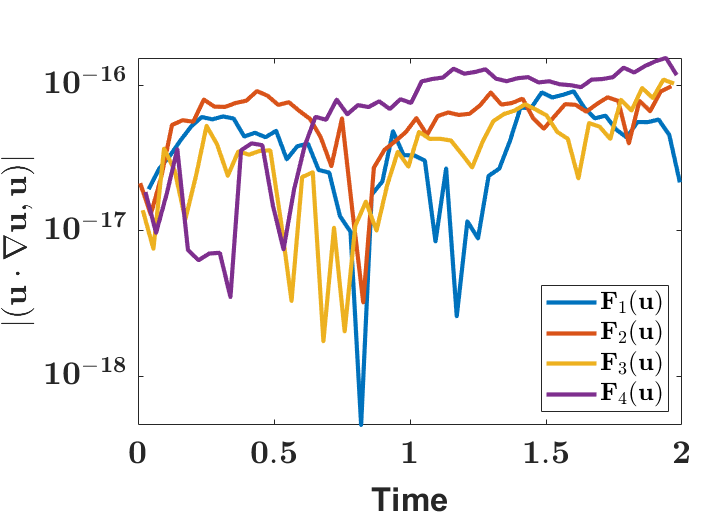}
			}
			\caption{Comparisons of the quantity $|(\overline{\bm{u}}^{n+\frac{1}{2}} \cdot \nabla\overline{\bm{u}}^{n+\frac{1}{2}},\bm{u}^{n+\frac{1}{2}})_h|$ between the CN2 scheme with various forms of $\bm{F}(\bm{u})$.}\label{innerCP}
		\end{figure} 
		
		Furthermore, a comprehensive numerical comparison is performed among our CN2 scheme, the DRLM scheme, and the ZEC scheme.  In Figure~\ref{CPOSC}, it can be seen that both the DRLM scheme and the ZEC scheme noticeably violate the original energy dissipation law under a relatively coarse time step. In contrast, our proposed scheme, equipped with the robust term $\bm{F}_4(\bm{u})$, accurately predicts the expected energy evolution across various large Reynolds numbers. These findings demonstrate that, under relatively coarse time step sizes, our developed scheme can yield substantially improved accuracy compared to the DRLM and ZEC approaches. The numerical results in Figure~\ref{largestepTGV} further validate that our scheme maintains stability even in long-time simulations with coarse time steps.
		\begin{figure}[H]
			\subfigure[$Re=1000$, the DRLM scheme.]{
				\includegraphics[width=2in]{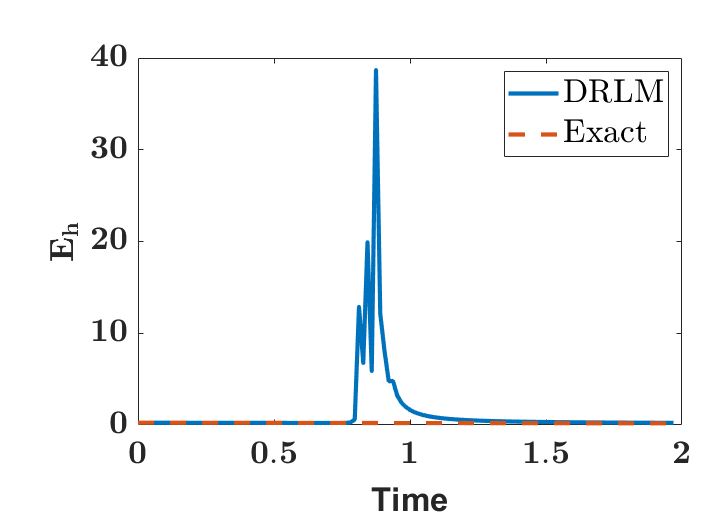}
			}
			\subfigure[$Re=1000$, the ZEC scheme.]{
				\includegraphics[width=2in]{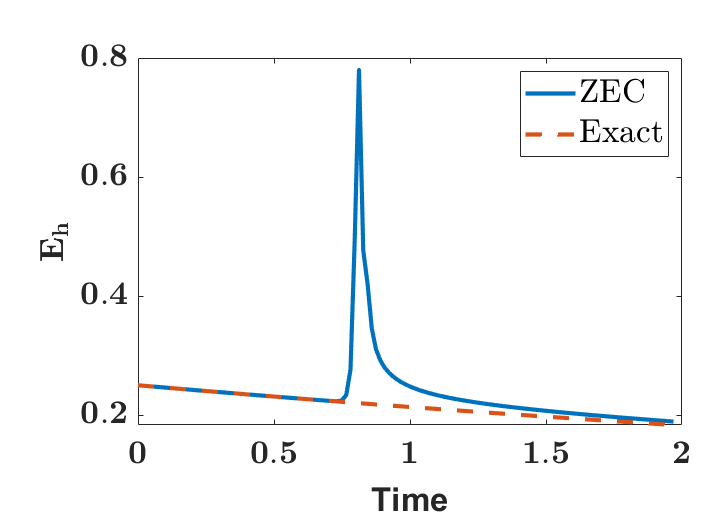}
			}
			\subfigure[$Re=1000$, CN2 scheme.]{
				\includegraphics[width=2in]{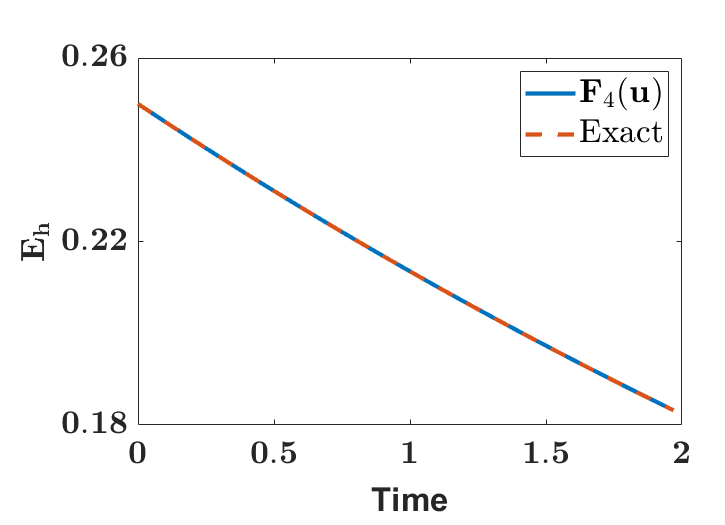}
			}\\
			\subfigure[$Re=10000$, the DRLM scheme.]{
				\includegraphics[width=2in]{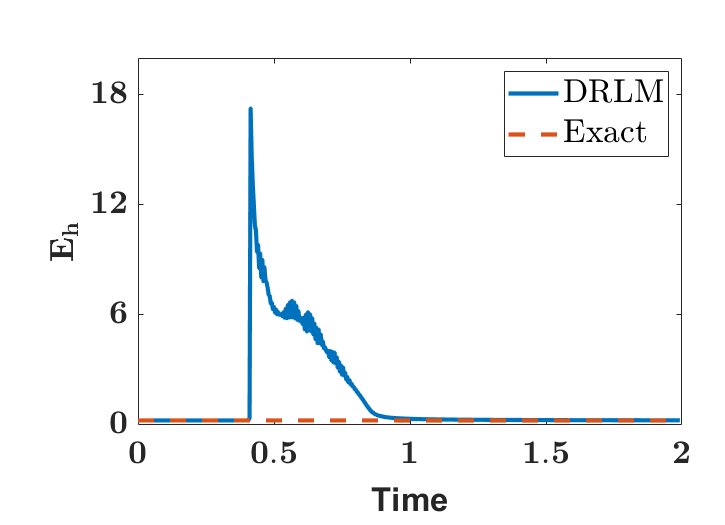}
			}
			\subfigure[$Re=10000$, the ZEC scheme.]{
				\includegraphics[width=2in]{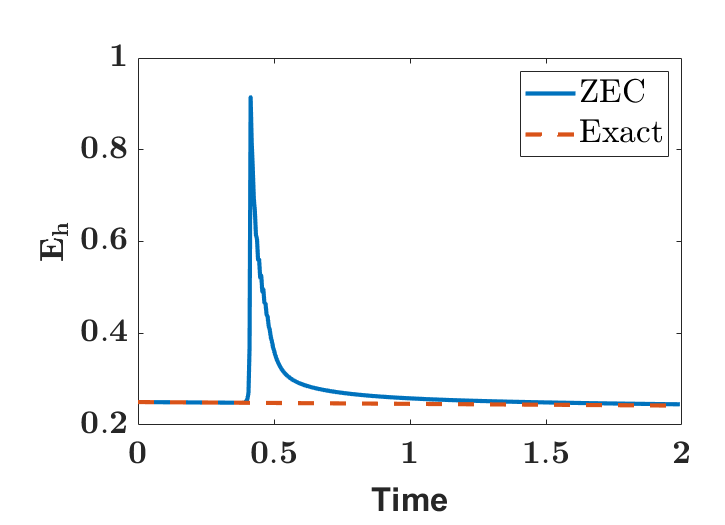}
			}
			\subfigure[$Re=10000$, CN2 scheme.]{
				\includegraphics[width=2in]{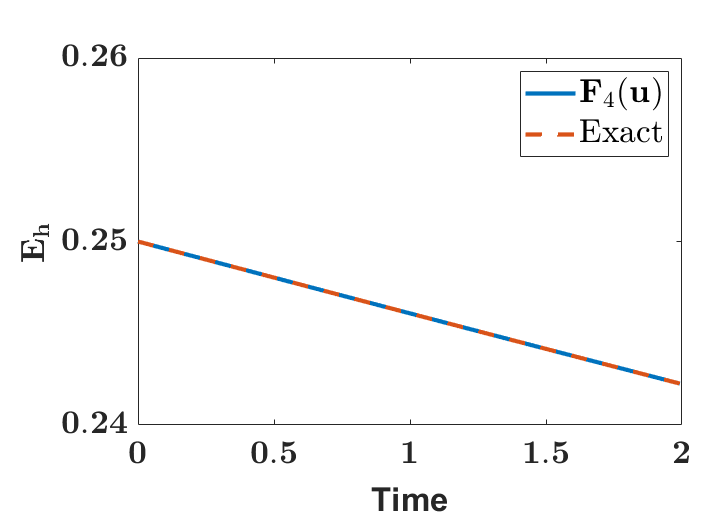}
			}
			\caption{Energy evolution for CN2 scheme with recommended $\bm{F}_4(\bm{u})$, DRLM scheme and ZEC scheme. Mesh size is $h=\frac{1}{128}$, $\tau=\frac{1}{64}$ for $Re=1000$, and $h=\frac{1}{256}$, $\tau=\frac{1}{256}$ for $Re=10000$.}\label{CPOSC}
		\end{figure} 
		
		\begin{figure}[ht]
			\centering
			\subfigure[$Re=1000$, $\tau=\frac{1}{64}$.]{
				\includegraphics[width=3in]{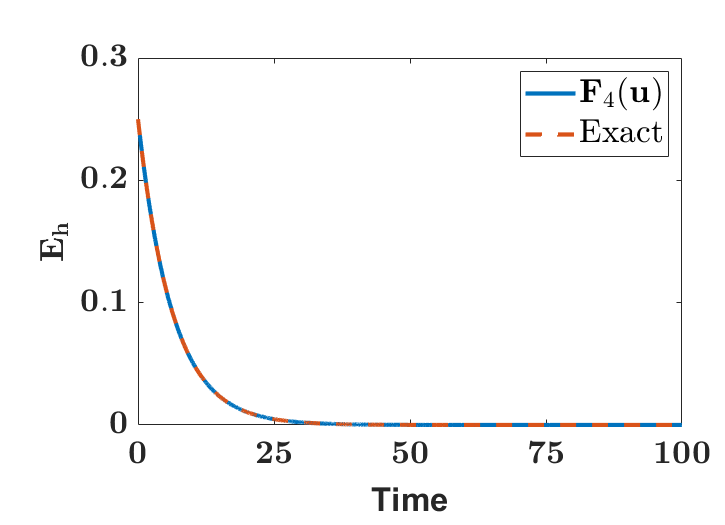}
			}
			\subfigure[$Re=10000$, $\tau=\frac{1}{256}$.]{
				\includegraphics[width=3in]{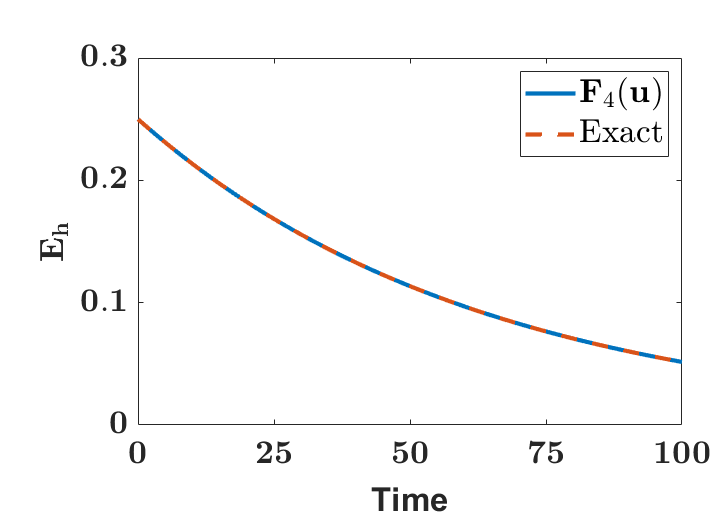}
			}
			\caption{Energy evolution for CN2 scheme with $\bm{F}_4(\bm{u})$. Mesh size is $h=\frac{1}{128}$, $\tau=\frac{1}{64}$ for $Re=1000$, and $h=\frac{1}{256}$, $\tau=\frac{1}{256}$ for $Re=10000$ in the long-time simulation.}\label{largestepTGV}
		\end{figure} 
		% 	\begin{figure}[H]
			%	\centering
			%			\subfigure[$Re=1000$, $\tau=\frac{1}{64}$ for various of $\bm{F}(\bm{u})$.]{
				%		\includegraphics[width=3in]{figure/taylorgreen/CCP50}
				%	}
			%	\subfigure[$Re=10000$, $\tau=\frac{1}{256}$ for various of $\bm{F}(\bm{u})$.]{
				%		\includegraphics[width=3in]{figure/taylorgreen/CCP53}
				%	}\\
			%		\subfigure[$Re=1000$, $\tau=\frac{1}{64}$ for DRLM, ZEC and our schemes.]{
				%		\includegraphics[width=3in]{figure/taylorgreen/CCP61}
				%	}
			%	\subfigure[$Re=10000$, $\tau=\frac{1}{256}$ for DRLM, ZEC and our schemes.]{
				%		\includegraphics[width=3in]{figure/taylorgreen/CCP64}
				%	}\\
			%		\subfigure[$Re=1000$, $\tau=\frac{1}{64}$, long-time simulation.]{
				%		\includegraphics[width=3in]{figure/taylorgreen/CCP54}
				%	}
			%	\subfigure[$Re=10000$, $\tau=\frac{1}{256}$, long-time simulation.]{
				%		\includegraphics[width=3in]{figure/taylorgreen/CCP55}
				%	}
			%		\caption{Energy evolution comparison for CN2 scheme with various forms of $\bm{F}(\bm{u})$, DRLM scheme and ZEC scheme. Mesh size is $h=\frac{1}{128}$ for $Re=1000$, and $h=\frac{1}{256}$ for $Re=10000$ under large time steps.}\label{largestepTGV}
			%\end{figure} 

			\subsection{Lid driven cavity flow}
			\label{Lid-driven square cavity flow}
			In this example, we further evaluate our algorithm's performance using the classical lid-driven cavity flow benchmark. All computations proceed until reaching a steady state, with convergence defined by the criterion $\| \bm{u}^n-\bm{u}^{n-1}\|_{\infty} \leq 10^{-6}$.
			
			\begin{figure}[H]
				\centering
				\subfigure[Velocity streamline ($Re=5000$).]{
					\includegraphics[width=3in]{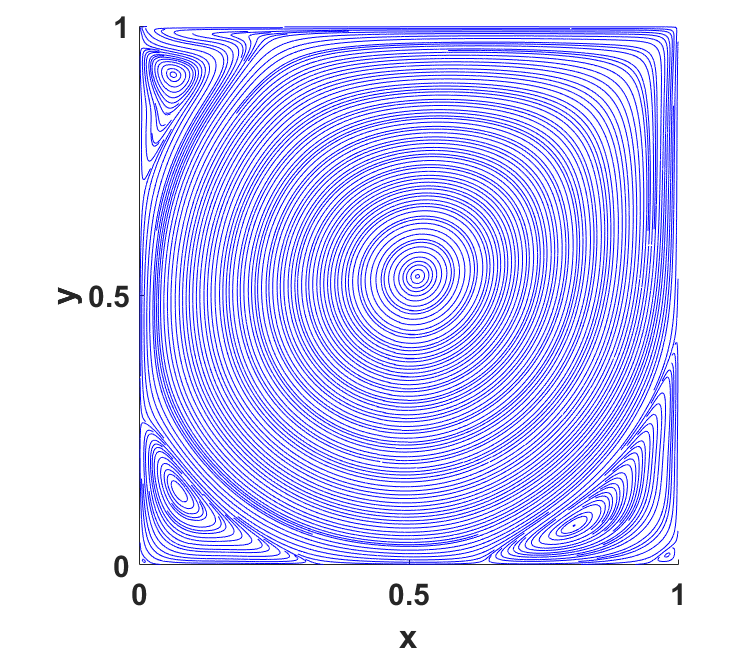}
				}
				\subfigure[Velocity magnitude ($Re=5000$).]{
					\includegraphics[width=3in]{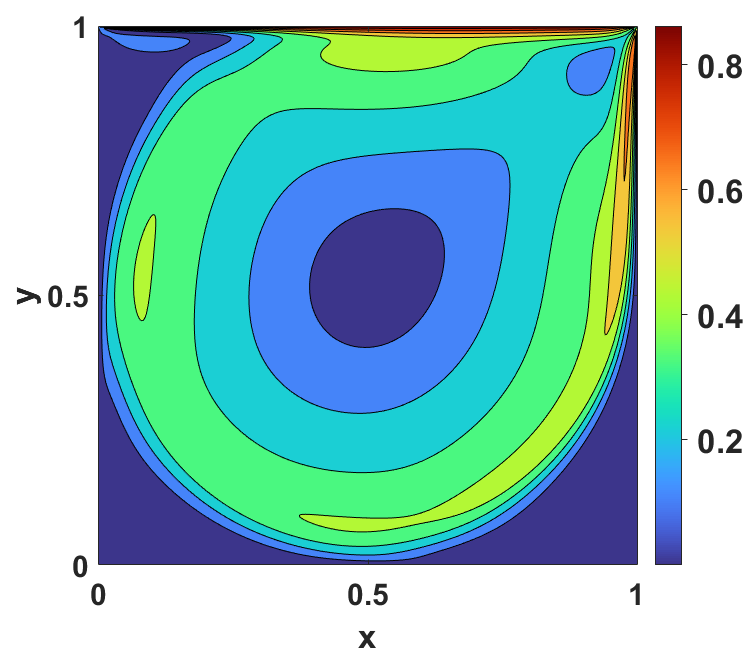}
				}
				\caption{Streamline (left) and contour plots (right) of the velocity magnitude at the steady state.}
				\label{LiddrivenCon2}
			\end{figure} 
			\begin{figure}[H]
				\centering
				\subfigure[Velocity-$u$ on $x=0.5$.]{
					\includegraphics[width=3in]{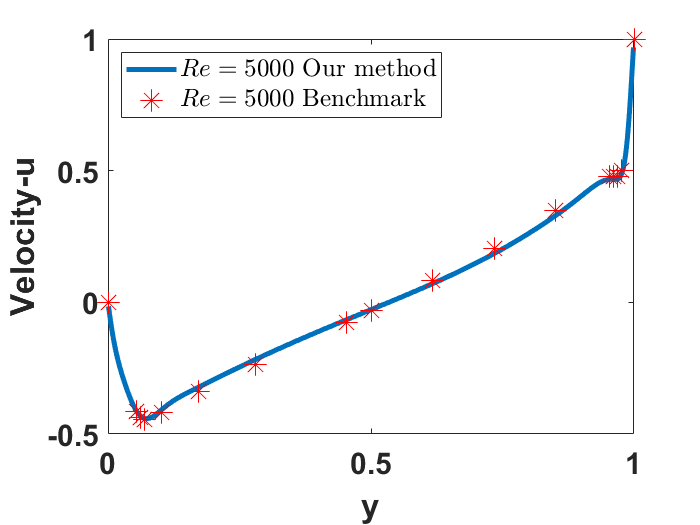}
				}
				\subfigure[Velocity-$v$ on $y=0.5$.]{
					\includegraphics[width=3in]{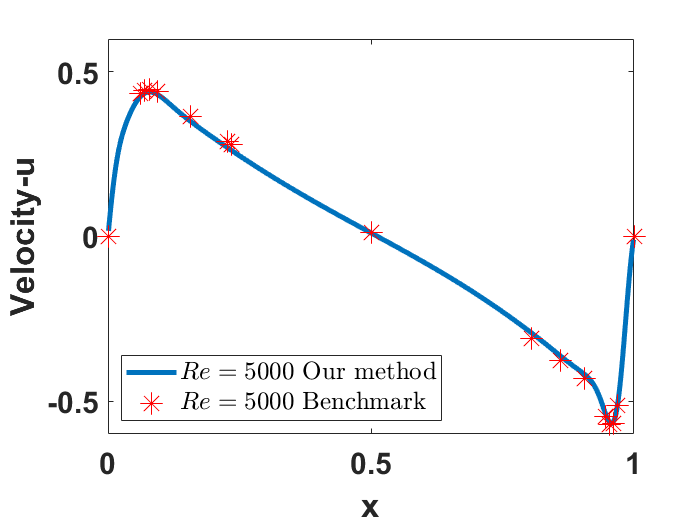}
				}
				\caption{Velocity on $x=0.5$ and $y=0.5$ of the lid-driven cavity flow.}\label{uvcut2}
			\end{figure} 
			For the 2D problem \cite{TamerAAEJ2017}, we consider the computational domain $\Omega = [0, 1]^2$, bounded by three stationary no-slip walls (at $x=0,~ x=1,$ and $y=0$) and a lid (at $y = 1$) moving with a tangential unit velocity. As no analytical solution is available for this problem, our numerical results are compared against established benchmark data reported in \cite{TamerAAEJ2017}. 
			The simulation is carried out with $Re=5000,~ h=\frac{1}{512}$, $\tau=4\times 10^{-4}$ and  $\bm{F}(\bm{u})=\bm{u}$. The streamline and contour plot of the velocity magnitude at the final steady state are shown in Figure~\ref{LiddrivenCon2}. In addition to the primary and secondary vortices in the lower corners, a third vortex is observed in the upper left corner. Furthermore, Figure~\ref{uvcut2} displays the velocity profiles along the vertical and horizonal centerlines of the cavity. These results demonstrate that the proposed scheme effectively captures the flow dynamics and accurately resolves the steady-state structure in line with existing reference solutions \cite{TamerAAEJ2017}.

			\begin{figure}[ht]
				\centering
				\subfigure[$\omega_1 (0.5,y,z)$.]{
					\includegraphics[width=3in]{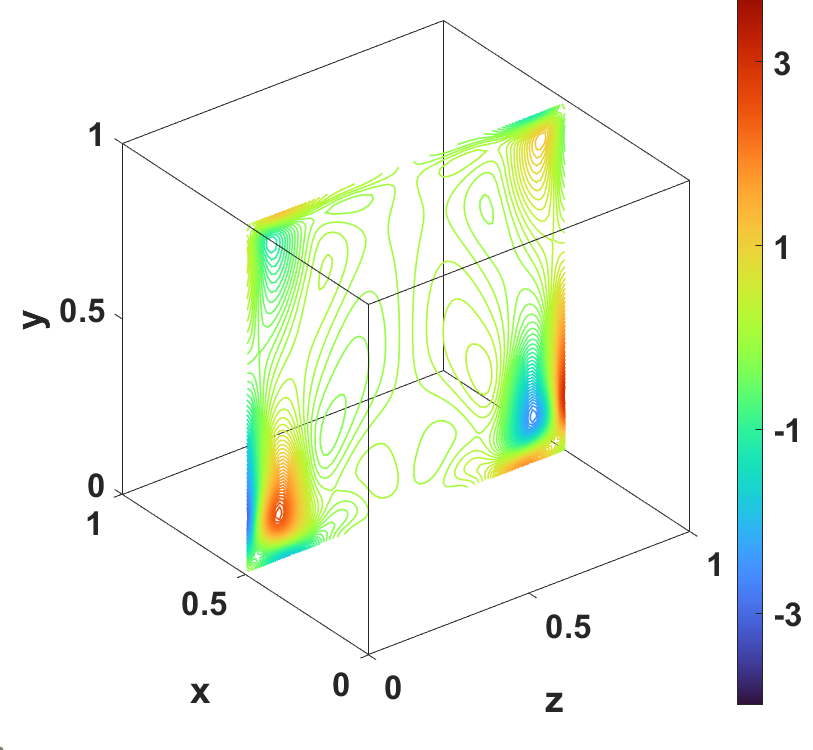}
				}
				\subfigure[$\omega_2 (x,0.5,z)$.]{
					\includegraphics[width=3in]{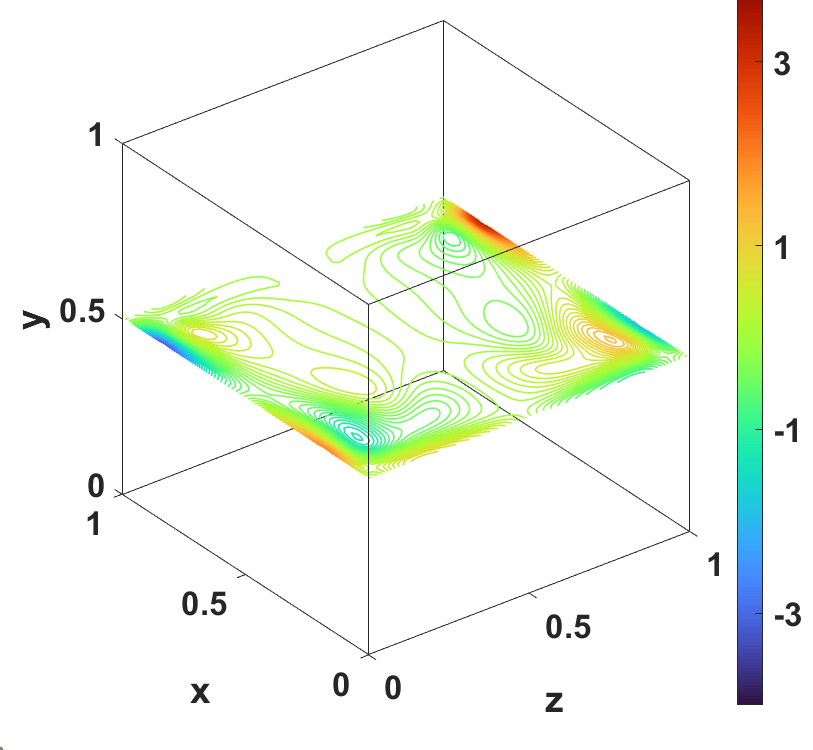}
				}\\
				\subfigure[$\omega_3 (x,y,0.5)$.]{
					\includegraphics[width=3in]{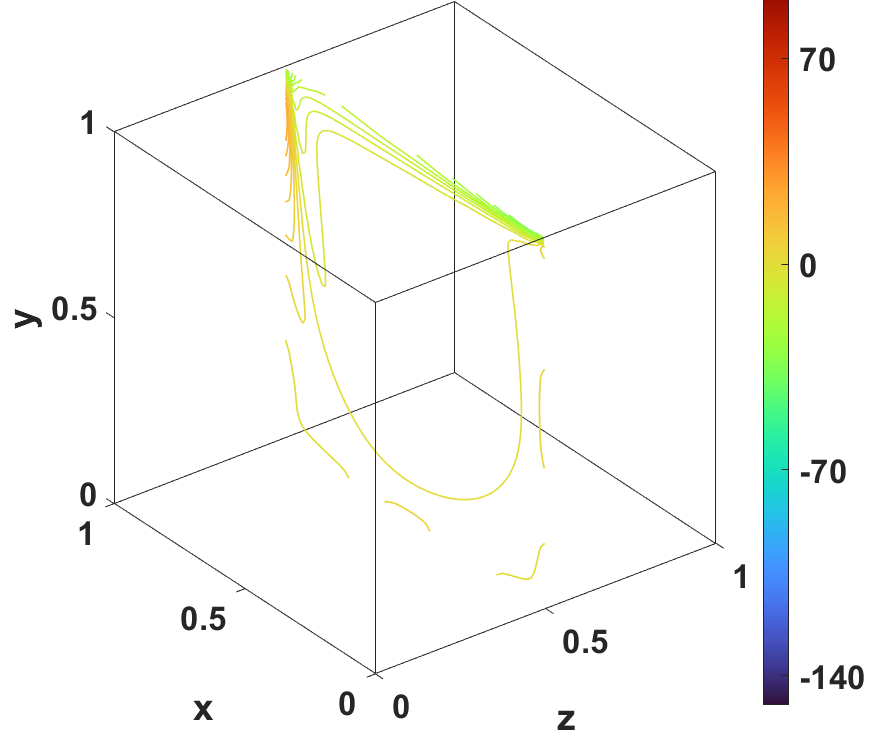}
				}
				\caption{Contour plots of the vorticity components on the midplanes at the steady state.}\label{555}
			\end{figure} 
			\begin{figure}[ht]
				\centering
				\subfigure[Velocity-$u$ on $x=z=0.5$.]{
					\includegraphics[width=3in]{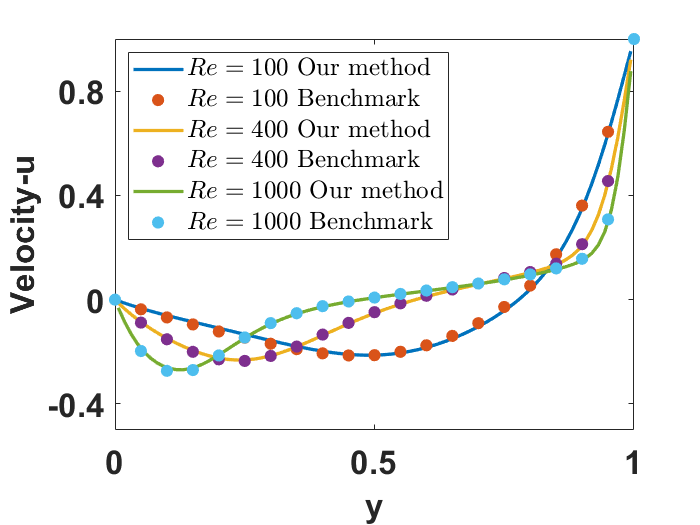}
				}
				\subfigure[Velocity-$v$ on $y=z=0.5$.]{
					\includegraphics[width=3in]{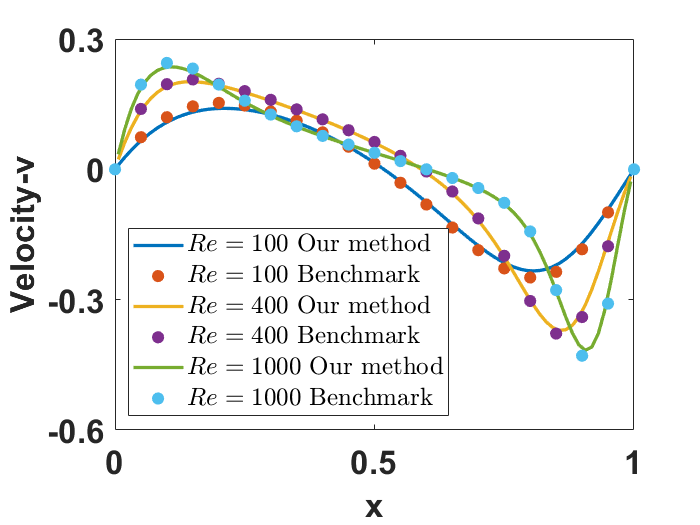}
				}
				\caption{Steady-state velocity profiles along the cavity centerlines.}\label{3Duvcut2}
			\end{figure}

			We next consider the 3D lid-driven flow \cite{ALBENSOEDER2005JCP} in a cubic cavity $\Omega=[0,1]^3$. The fluid is initially at rest, movement is induced by the top lid (at $y=1$) being translated at a constant unit velocity in the positive $x-\mathrm{direction}$. All remaining boundaries are subject to homogeneous Dirichlet boundary conditions. Simulations are performed at $Re=100,~400,~1000$, $h=\frac{1}{80}$, $\tau=5\times 10^{-4}$, and the flow is evolved using our proposed scheme until a steady state is achieved. Figure~\ref{555} displays contour plots of the vorticity components $\boldsymbol{\omega} = (\omega_1,\omega_2,\omega_3)^T = \nabla \times \bm{u}$ on the central planes $x=0.5,~y=0.5,$ and $z=0.5$. The flow exhibits symmetric about the plane $z=0.5$, and in addition to the corner vortices, Taylor-G$\ddot{\mathrm{o}}$rtler-like vortices are observed near the bottom region of the cavity. These features are consistent with previous observations reported in \cite{JIANG1994CMAME}. Furthermore, Figure~\ref{3Duvcut2} illustrates the $u$- and $v$-velocity profiles along the vertical ($x=z=0.5$) and horizontal ($y=z=0.5$) centerlines at steady-state. The results exhibit excellent agreement with the benchmark data presented in \cite{ALBENSOEDER2005JCP}, thereby validating the accuracy and robustness of the proposed scheme in capturing complex three-dimensional flow structures.

			\subsection{Kelvin-Helmholtz instability}
			When there is an initial velocity difference across a shear layer, small disturbance can grow over time, leading to the formation of vortices. This phenomenon is known as the Kelvin-Helmholtz instability. 
			
			In this example, we consider the 2D incompressible NS equations in $\Omega=[0,1]^2$ without the external force. Slip boundary conditions, given by $\frac{\partial u}{\partial y}=0$ and $v=0$, are imposed along the top and bottom boundaries, while periodic boundary conditions are applied at the left and right boundaries. The initial condition reads as \cite{Phipp2019CMA}
			\begin{align*}
				&u(x,y,0)=u_{\infty}\tanh \left( \frac{2y-1}{\delta_0}\right) +c_n\partial_y \psi(x,y),\\
				&v(x,y,0)=-c_n\partial_x \psi(x,y),
			\end{align*}
			with the stream function
			\begin{align*}
				\psi(x,y)=u_{\infty}\mathrm{exp}\left(-\frac{(y-0.5)^2}{\delta^2_0}\right)[\cos(8\pi x)+\cos(20\pi x)],
			\end{align*}
			where $\delta_0=1/28$ is the initial vorticity thickness, and $u_{\infty}=1$ is a reference velocity. $c_n=10^{-3}$ is a scaling factor. Numerical simulations are conducted with $h=\frac{1}{256}$, $\Delta t=\frac{1}{420}$, and $\bm{F}(\bm{u})=\bm{u}$, using a fixed viscosity $\nu=\frac{1}{2800}$. The characteristic time unit is defined as $\bar{t}=\frac{1}{28}$, and computations are carried out up to the final time $T=200\bar{t}$. Figure~\ref{444} displays the time evolution of the vorticity field. It is clearly observed that Four vortices emerge from the initial shear layer, merge into two dominant structures, and persist with diminishing strength until $T=200\bar{t}$, consistent with the results reported in \cite{doan4938899dynamically}. Moreover, Figure~\ref{KHIen} proves the energy dissipation property of the system, in agreement with the reference energy curve obtained using our scheme with a small time step $\tau=\frac{1}{2500}$. Additionally, the evolution of the quantity $|(\overline{\bm{u}}^{n+\frac{1}{2}} \cdot \nabla\overline{\bm{u}}^{n+\frac{1}{2}},\bm{u}^{n+\frac{1}{2}})_h|$ remains relative small number, consistent with the former observation in the energy dissipation test.

			\begin{figure}[H]
				\centering
				\subfigure[$t=10\bar{t}$.]{
					\includegraphics[width=2.5in]{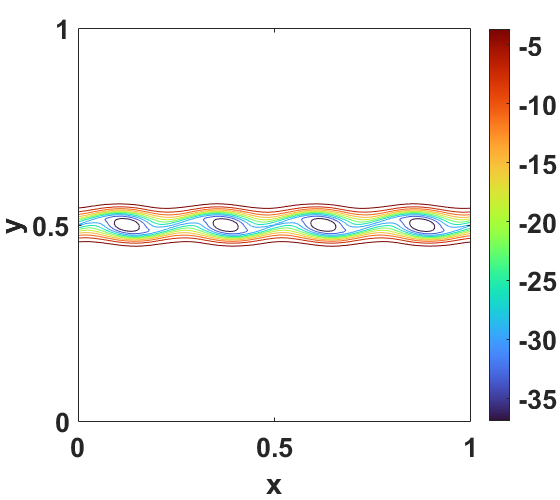}
				}
				\subfigure[$t=17\bar{t}$.]{
					\includegraphics[width=2.5in]{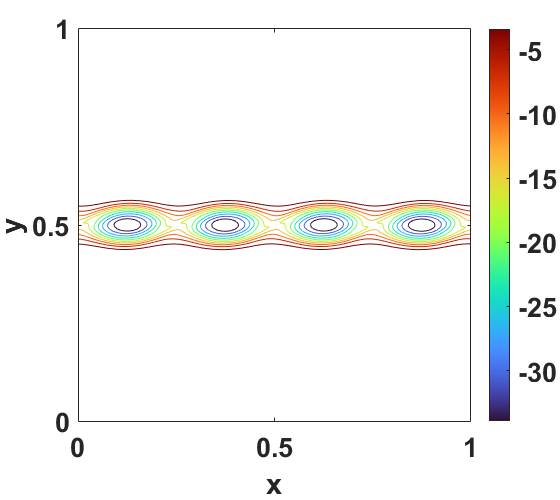}
				}\\
				\subfigure[$t=42\bar{t}$.]{
					\includegraphics[width=2.5in]{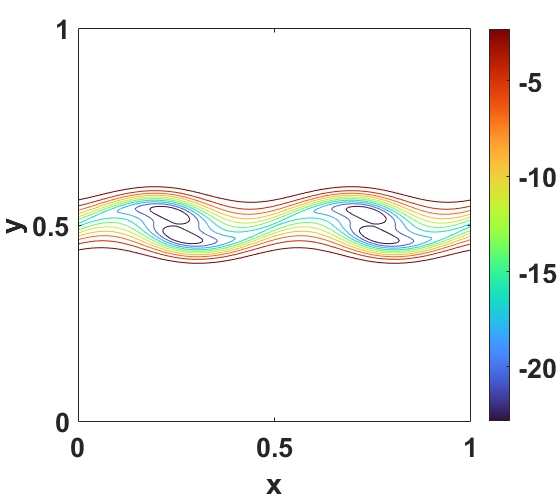}
				}
				\subfigure[$t=50\bar{t}$.]{
					\includegraphics[width=2.5in]{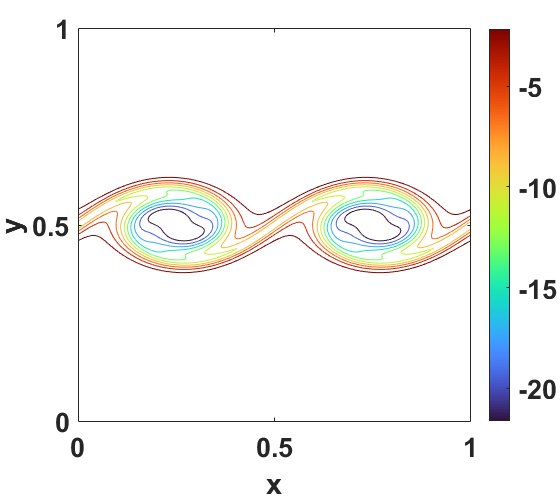}
				}\\
				\subfigure[$t=140\bar{t}$.]{
					\includegraphics[width=2.5in]{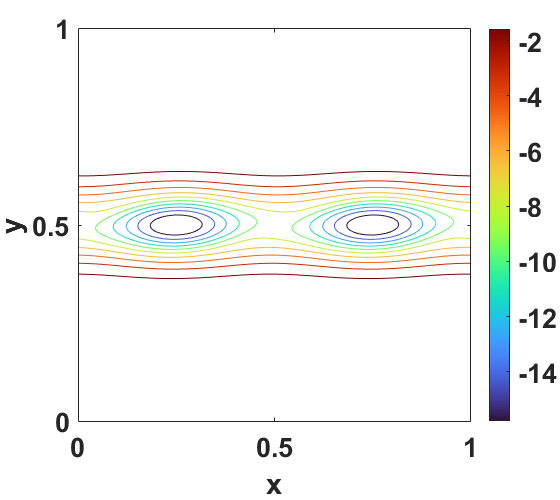}
				}
				\subfigure[$t=200\bar{t}$.]{
					\includegraphics[width=2.5in]{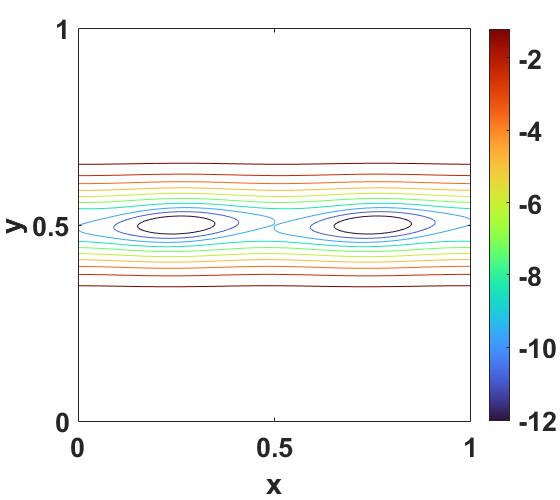}
				}
				\caption{Vorticity field at various times.}\label{444}
			\end{figure} 
			\begin{figure}[H]
				\centering
				\subfigure[Energy evolution.]{
					\includegraphics[width=3in]{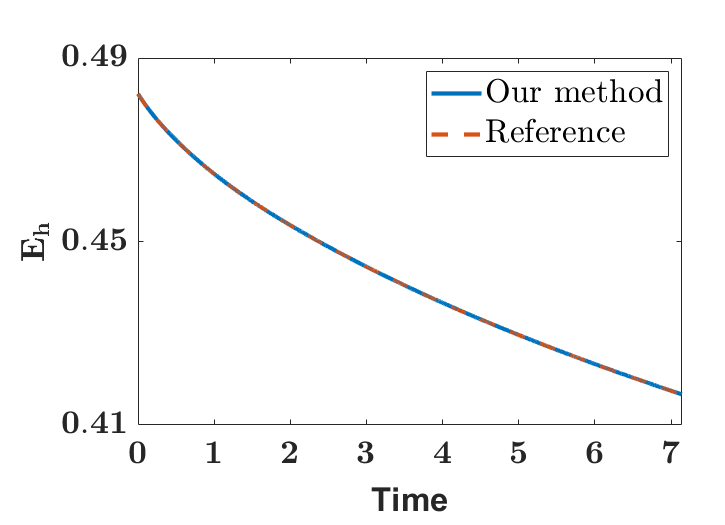}
				}
				\subfigure[The quantity $|(\overline{\bm{u}}^{n+\frac{1}{2}} \cdot \nabla\overline{\bm{u}}^{n+\frac{1}{2}},\bm{u}^{n+\frac{1}{2}})_h|$. ]{
					\includegraphics[width=3in]{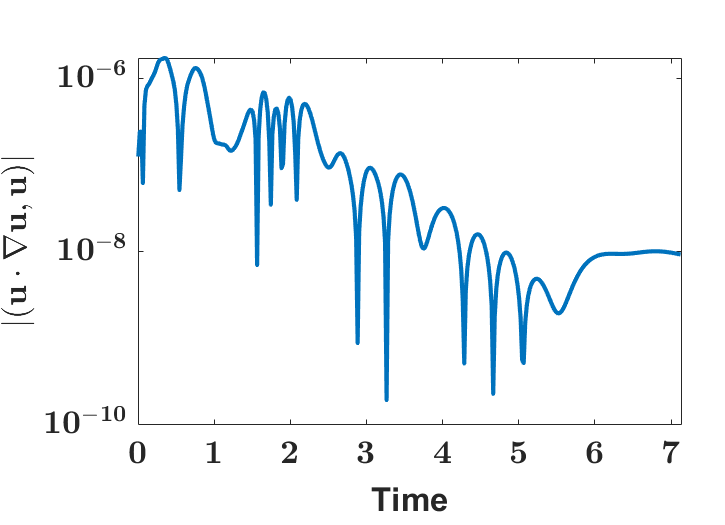}
				}
				\caption{Energy evolution for CN2 scheme with $h=\frac{1}{256}$ and $\frac{1}{420}$ and the quantity $|(\overline{\bm{u}}^{n+\frac{1}{2}} \cdot \nabla\overline{\bm{u}}^{n+\frac{1}{2}},\bm{u}^{n+\frac{1}{2}})_h|$ .}\label{KHIen}
			\end{figure}

			\section{Conclusion}
			\label{Conclusion}
			
			In this paper, we have proposed the robust reformulation for the incompressible NS equations to develop first- and second-order structure-preserving algorithms. In contrast to existing techniques such as SAV or Lagrangian methods, our approach eliminates the need for auxiliary variables yet facilitates an explicit discretization of the convective term, strictly conserving the original energy dissipation law. Through spatial discretization on a staggered grid, we develop fully discrete schemes which are shown to conserve the original energy dissipation law and exhibit existence and uniqueness of solutions. Our numerical schemes are linear and easy to implement, as they only require solving three generalized Stokes systems and a $2\times 2$ system of linear equations at each time step. Various numerical experiments and comparisons have been conducted to verify the accuracy, energy stability, and the efficiency of our proposed schemes. Numerical experiments show that as $|(\overline{\bm{u}}^{n+\frac{1}{2}} \cdot \nabla\overline{\bm{u}}^{n+\frac{1}{2}},\bm{u}^{n+\frac{1}{2}})_h|$ approaches zero, computational accuracy improves. Based on this observation, we propose developing adaptive algorithms using this metric, with implementation details to follow in future work.
			
			In addition, we note that by combining our robust reformulation with projection techniques, more efficient decoupled energy stable schemes can be developed. For example, first- and second-order pressure-correction formulations are presented below
			\begin{equation}
				\begin{cases}
					\frac{\tilde{\bm{u}}^{n+1}-\bm{u}^n}{\tau} -\nu \Delta \tilde{\bm{u}}^{n+1} + 	\bm{B}(\bm{u}^n,\tilde{\bm{u}}^{n+1}) + \nabla p^{n}=\bm{f}^{n+1}, \quad \tilde{\bm{u}}|_{\partial \Omega}=\bm{0}, \\
					\frac{\bm{u}^{n+1} - \tilde{\bm{u}}^{n+1}}{\tau} + \nabla (p^{n+1}-p^{n}) = 0,\\
					\nabla \cdot \bm{u}^{n+1}=0, \quad \bm{u}^{n+1}\cdot \bm{n}|_{\partial \Omega} = 0,
				\end{cases}
			\end{equation}
			and
			\begin{equation}
				\begin{cases}
					\frac{3\tilde{\bm{u}}^{n+1}-4\bm{u}^n+\bm{u}^{n-1}}{2\tau} -\nu \Delta \tilde{\bm{u}}^{n+1} + 	\bm{B}(\overline{\bm{u}}^{n+1},\tilde{\bm{u}}^{n+1}) + \nabla p^{n}=\bm{f}^{n+1}, \quad \tilde{\bm{u}}|_{\partial \Omega}=\bm{0}, \\
					\frac{3\bm{u}^{n+1} - 3\tilde{\bm{u}}^{n+1}}{2\tau} + \nabla (p^{n+1}-p^{n}) = 0,\\
					\nabla \cdot \bm{u}^{n+1}=0, \quad \bm{u}^{n+1}\cdot \bm{n}|_{\partial \Omega} = 0.
				\end{cases}
			\end{equation}
			Furthermore, the robust reformulation proposed in this paper can also be extended to hydrodynamic phase field models for developing efficient structure-preserving algorithms. These aspects will be elaborated upon in detail in our subsequent work.
			
			\section*{Acknowledgements}
			This work is supported by the National Natural Science Foundation of China (Grant Nos. 12271252 and 12201297), the Natural Science Foundation of Jiangsu Province (Grant No. BK20220131), the Fundamental Research Funds for the Central Universities (Grant No. NE2024009) and the National Science and Technology Major Project (Grant No. J2019-II-0007-0027).

			%\section*{Appendix: finite difference discretization on a staggered grid}

			%%%% Bibliography  %%%%%%%%%%
			\bibliographystyle{elsarticle-num} %/unsrt
			\bibliography{Refs}

		\end{document}